\documentclass{amsart}%
\usepackage{amsfonts}
\usepackage[normalem]{ulem}
\usepackage[font=small]{caption}
\usepackage{mathtools,graphicx}
\usepackage{hyperref}
\usepackage{adjustbox}
\usepackage{rotating}
\usepackage{makecell}
\usepackage{latexsym}
\usepackage{subcaption}
\usepackage{hyperref}
\usepackage{animate}
\usepackage[T1]{fontenc}
\usepackage{ifthen}
\usepackage{pgfmath}
\usepackage[margin=1cm]{geometry}
\usepackage{amsmath}
\usepackage{mathrsfs}
\usepackage{pstricks}
\usepackage{listings}
\usepackage{hyperref}
\usepackage{enumitem}
\usepackage{csvsimple}
\usepackage{todonotes}
\usepackage{amssymb,bm}
\usepackage{amsmath}
\usepackage{tikz-3dplot}
\usepackage{amsthm}
\usepackage{tikz}
\usepackage{tikz-cd}
\usepackage{epsfig}
\usepackage{verbatim}
\usepackage{float}
\usepackage{tabularx}
\usepackage{algorithmic}
\usepackage{amssymb}
\usepackage{graphicx}%
\providecommand{\U}[1]{\protect\rule{.1in}{.1in}}
\definecolor{codegray}{rgb}{0.5,0.5,0.5}
\lstdefinestyle{mystyle}{
commentstyle=\color{codegreen},
keywordstyle=\color{magenta},
numberstyle=\tiny\color{codegray},
stringstyle=\color{codepurple},
basicstyle=\ttfamily\footnotesize,
breakatwhitespace=false,
breaklines=true,
captionpos=b,
keepspaces=true,
numbers=left,
numbersep=5pt,
showspaces=false,
showstringspaces=false,
showtabs=false,
tabsize=2
}  
\lstdefinelanguage{GAP}{
basicstyle=\ttfamily,
keywords={true, false, function, return, fail, if, in, while, do, od, else, elif, fi, break, continue},
keywordstyle=\color{blue}\bfseries,
otherkeywords={      >, <, ==
},
breaklines=true,
identifierstyle=\color{black},
sensitive=True,
comment=[l]{\#},
commentstyle=\color{cyan},
stringstyle=\color{red},
morestring=[b]',
morestring=[b]"
}  
\providecommand{\U}[1]{\protect\rule{.1in}{.1in}}

\newcolumntype{Y}{>{\raggedleft\arraybackslash}X}

\def\bz{{\mathbb{Z}}}

\def\vs{\vskip.3cm}
\def\noi{\noindent}

\def\gdeg{G\text{\rm -deg}}

\def\Om{\Omega}

\newtheorem{theorem}{Theorem}

\newtheorem{corollary}[theorem]{Corollary}

\newtheorem{definition}[theorem]{Definition}
\newtheorem{example}[theorem]{Example}

\newtheorem{lemma}[theorem]{Lemma}

\newtheorem{proposition}[theorem]{Proposition}
\newtheorem{remark}[theorem]{Remark}

\def\bz{\mathbb Z}
\def\vs{\vskip.3cm}

\def\vp{\varphi}

\def\bz{{\mathbb{Z}}}

\def\vs{\vskip.3cm}
\def\noi{\noindent}

\def\gdeg{G\text{\rm -deg}}

\def\Om{\Omega}

\def\cV{\mathcal V}
\def\cW{\mathcal W}

\def\cU{\mathcal U}

\def\id{\text{\rm Id}}

\begin{document}
\title{Nonlinear Vibrational Mode of Molecule with Octahedral Configuration}
\author[J. Liu]{Jingzhou Liu}
\address{Department of Mathematical Sciences the University of Texas at Dallas
Richardson, 75080 USA}
\email{Jingzhou.Liu@UTDallas.edu}
%\author[W. Krawcewicz]{Wieslaw Krawcewicz}
%\address{Department of Mathematical Sciences the University of Texas at Dallas
%Richardson, 75080 USA}
%\email{wieslaw@utallas.edu}

\maketitle
\begin{abstract}
In this work, we investigate the nonlinear dynamics of molecules with an octahedral
configuration, with particular focus on sulfur hexafluoride $SF_6.$ 
Under the assumption of isotypic nonresonance, we apply the method of equivariant gradient degree to prove the existence of
branches of periodic solutions emerging from the critical orbit of equilibrium, corresponding to at
least 16 distinct types of symmetries with maximal orbit kinds. %presenting a full topological classification of their spatial-temporal symmetries. %These branches appear as standing or rotating waves propagating along the molecule, exhibiting triangular, tetragonal, and hexagonal symmetries.
Numerical animations are presented to illustrate the detected vibrational modes.
\end{abstract}
\noi \textbf{AMS subject classications:} 37G40, 47J15, 47H11, 37G15,  37J46, 70H33 

\medskip

\noi \textbf{Key Words:} symmetric bifurcation, nonlinear dynamics, equivariant
gradient degree, octahedral molecule.

\section{Introduction}
Understanding the dynamic behavior of molecular systems provides detailed insights into their intrinsic properties. To this end, various mathematical models have been developed for different molecules, many of which are grounded in classical mechanics. \cite{Nanotube} introduced the Morse bond potential, harmonic cosine angle potential, and 2-fold torsion as the energy model for carbon nanotubes, and used them to study its hydrophobic behavior; \cite{tetraphosphorus} explored the stability and nonlinear modes of  tetraphosphorus;  \cite{jing2009molecular, weeks1989rotation, wu1987vibrational} proposed force fields of fullerene that only consider small fluctuations. Many of these have been calibrated to replicate the vibrational modes observed in experimental techniques such as infrared and Raman spectroscopy, which are adept at detecting linear vibrational modes. However, to fully capture the complexity of molecular dynamics, alternative mathematical tools are needed, especially for analyzing nonlinear vibrational phenomena that lie beyond the reach of current experimental methods. \vs

In this work, we explore molecules with octahedral configuration, such as $SF_6\;$, where a central atom is surrounded by six ligand atoms (see Figure \ref{fig:oct}). To reduce complexity, we assume that the vibrations of the central atom are negligible and hereinafter, denoted by $\mathcal O$ its spatial position. To be more precise, we consider the following Newtonian system
\begin{equation}
\ddot{v}(t)=-\nabla U(v(t)), \label{eqn01}%
\end{equation}
where $v(t)=(v_{1}(t),\,v_{2}(t),\,v_{3}(t),\,v_{4}(t),v_5(t),v_6(t))^T.$
Each $v_{j}(t)\in
\mathbb{R}^{3}$ stands for spatial position of individual ligand atom at time $t$ and $v_j(t)$ interacts with both central atom $\mathcal O$ and all other $v_{k}(t)$ for $k\not =j$. 
Moreover, the potential $U:\mathbb R^3 \to \mathbb R$ in \eqref{eqn01} is given by  
\[
U(v):=\sum_{1\leq j<k\leq6}U_1(|v_{j}-v_{k}|^2)+\sum_{j=1}^{6}U_2(|v_{j}-\mathcal O|^2).
\]
where 
\begin{equation}\label{eq:U1}
U_1(r)= \frac{\sigma_1}{r^{6}}-\frac{\sigma_2}{r^{3}}  +\frac{\sigma_3}{\sqrt{r}},\;\;U_2(r)=(\sqrt{r}-1)^2,\;\; r> 0
\end{equation}
Here $U_1$ describes the combination of Lennard Jones and Columnb Potential among the interaction of $F$-$F$ atoms, and $U_2$ is the bond stretching between $F$-$\mathcal O$. One can refer to  \cite{strauss1994high},\cite{olivet2007},\cite{dellis} for more details about force field of $SF_6.$ %It is important to note that  in this work, the $F–S–F$ angle bending is omitted due to the assumption of fixed central atom.
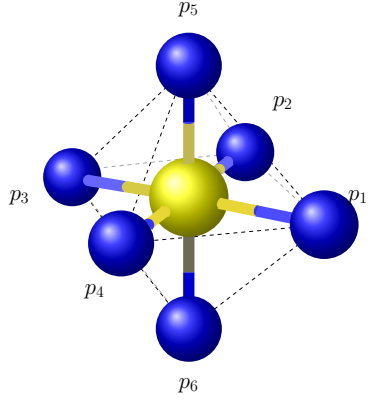
\begin{figure}[h]
\begin{center}
\scalebox{.5}
{\begin{tikzpicture}
\draw[dashed, color=black] (-3.1,0.533)--(-1.8,-1.23);
\draw[dashed, color=black] (0,3.7)--(-1.8,-1.23);
\draw[dashed, color=black] (3.6,-0.73)--(-1.8,-1.23);
\draw[dashed, color=black] (3.6,-0.73)--(0,-3.5);
\draw[dashed, color=gray!70] (3.6,-0.73)--(1.5,1.2);
\draw[dashed, color=black] (3.6,-0.73)--(0,3.5);
\draw[dashed, color=black] (-3.1,0.533)--(0,3.5);
\draw[dashed, color=black] (-1.8,-1.23)--(0,-3.5);
\draw[dashed, color=gray!70] (-3.1,0.533)--(0,-3.5);
\draw[dashed, color=gray!70] (-3.1,0.7)--(1.5,1.2);
\draw[dashed, color=gray!70] (0,3.5)--(1.5,1.2);
\draw[yellow!80!black, line width=3mm](-1.7,0.2833)--(0,0);
\shade[ball color=blue, opacity=1](1.5,1.2)circle (22pt);
\draw[blue!70!black, line width=3mm](0,4)--(0,2);
\shade[ball color=blue, opacity=1](0,3.5) circle (25pt);
\draw[yellow!30!black, line width=3mm](0,-0.9)--(0,-2);
\draw[yellow!70!black, line width=3mm](0.65,0.6291)--(0.8,0.7466);
\shade[ball color=yellow, opacity=1](0,0)circle (30pt); 

\draw[yellow!70!black, line width=3mm](0,2)--(0,0.9);
\draw[fill=blue!80!white,
draw=none] 
    (-0.16,2) arc[start angle=-180, end angle=0, x radius=0.15, y radius=0.05];
\draw[blue!80!black, line width=3mm](0,-2)--(0,-2.8);
\shade[ball color=blue, opacity=1](0,-3.5) circle (25pt); 
\draw[fill=yellow!70!black, draw=none] 
    (-0.22,0.92) arc[start angle=-180, end angle=0, x radius=0.15, y radius=0.04];
    
\draw[yellow!87!black, line width=3mm](-1.2,-0.82)--(-0.6,-0.35);
    \begin{scope}[shift={(-0.7,-0.35)}, rotate=-140]
  \draw[fill=yellow!87!black, draw=none]
    (0,-0.1) arc[start angle=270, end angle=90, radius=0.167];
\end{scope}
\begin{scope}[shift={(-1.73,-0.7466)}, rotate=-140]
  \draw[fill=blue!80!white, draw=none]
    (0,-0.1) arc[start angle=270, end angle=90, radius=0.15];
\end{scope}
\begin{scope}[shift={(-1.13,-0.6866)}, rotate=-140]
  \draw[fill=blue!80!white, draw=none]
    (0,-0.1) arc[start angle=270, end angle=90, radius=0.167];
\end{scope}
\draw[blue!60!white, line width=3mm](-1.2, -0.82)--(-1.5, -1.05);
    \begin{scope}[shift={(-1.25,-0.8)}, rotate=-140]
  \draw[fill=blue!60!white, draw=none]
    (0,-0.1) arc[start angle=270, end angle=90, radius=0.15];
\end{scope}
\draw[blue!60!white, line width=3mm](0.8,0.7466)--(1.1,0.9816);
    \begin{scope}[shift={(1,0.9811)}, rotate=-140]
  \draw[fill=blue!60!white, draw=none]
    (0,-0.1) arc[start angle=270, end angle=90, radius=0.15];
    \end{scope}
\begin{scope}[shift={(0.78,0.7766)}, rotate=-140]
  \draw[fill=yellow!80!black, draw=none]
    (0,-0.1) arc[start angle=270, end angle=90, radius=0.15];
\end{scope}
\shade[ball color=blue, opacity=1] (-3.1,0.533) circle (22pt);
\draw[blue!60!white, line width=3mm](-1.7,0.2833)--(-2.7,0.45);
\begin{scope}[shift={(-2.8,0.336)}, rotate=70]
  \draw[fill=blue!60!white, draw=none]
    (0,-0.08) arc[start angle=180, end angle=0, x radius=0.15, y radius=0.14];
\end{scope}
\shade[ball color=blue, opacity=1](-1.8,-1.23) circle (25pt);
\draw[yellow!87!black, line width=3mm](1.8,-0.365)--(0.8,-0.1622);
\draw[blue!70!white, line width=3mm](3.6,-0.73)--(1.8,-0.365);
\begin{scope}[shift={(0.71,-0.25)}, rotate=68]
  \draw[fill=yellow!87!black, draw=none]
    (0,-0.08) arc[start angle=180, end angle=0, x radius=0.1375, y radius=0.08];
\end{scope}
\begin{scope}[shift={(1.7,-0.465)}, rotate=75]
  \draw[fill=blue!80!white, draw=none]
    (0,-0.08) arc[start angle=180, end angle=0, x radius=0.14, y radius=0.08];
\end{scope}
    \begin{scope}[shift={(-1.78,0.18)}, rotate=62]
     \draw[fill=yellow!80!black, draw=none]
    (0,-0.08) arc[start angle=180, end angle=0, x radius=0.15, y radius=0.14];
    \end{scope}
\shade[ball color=blue, opacity=1](3.6,-0.73) circle (26pt);

%\draw[blue, line width=5mm](-1.75,0)--(-2.5,0);
\draw[fill=blue, draw=none] 
    (-0.235,-2.77) arc[start angle=-10, end angle=0, x radius=0.22, y radius=0.08];
\node at (-2.5,-2.5){\huge $p_4$};
\node at (0,-5){\huge $p_6$};
\node at (4.5,0){\huge $p_1$};
\node at (0,5){\huge $p_5$};
\node at (-4.5,0){\huge $p_3$};
\node at (2.5,2.5){\huge $p_2$};
                    % \rput(1,-2){\Large\bf $4$}
\end{tikzpicture}
}
\caption{Symmetry of Octahedral Molecule}
\label{fig:oct}
\end{center}
\end{figure}

By construction, the force field ensures that the potential $U$ in equation \eqref{eqn01} is invariant under the action of the group
\[
G=O(2)\times S_{4}\times \mathbb Z_2\times O(3),
\]
where $O(2)$ stands for temporal phase shift and reflection, $S_{4}\times \mathbb Z_2$ represents permutations of 6 ligand atoms, $O(3)$ the rotation and reflection of a single atom in space $\mathbb R^3$. Notice that this property enables us to employ equivariant methods for the study of molecule dynamics. 
%To fix the the center of mass, which, in this case, coincides with the central atom $\mathcal O$ at the origin, one can define a subspace of $\mathbb R^{18}$, given by
%\[
% V':=\{(v_1,v_2,v_3,v_4,v_5,v_6)^T\in \mathbb R^{18}:v_1+v_2+v_3+v_4+v_5+v_6=0\},
%\]
%from this we can exclude the collision orbits 
Next, we restrict system \eqref{eqn01} to the collision-free set
\[
\Omega_o:=\{(v_1,v_2,v_3,v_4,v_5,v_6)^T\in \mathbb R^{18}:v_{k}
\not =v_{j}\text{ for }k\not =j\},
\]
thereby excluding the collision orbits.
 %and the local energy minimum can be derived at the stationary point $v^0\in \mathbb R^{18},\;
%\nabla U(v^0)=0.$ \vs 
The local energy minimizer $v^0$ is identified as the stationary point satisfying $\nabla U (v^0) = 0$ and takes the form
\[
v^o=r^op^o,\quad p^o=(p_1,p_2,\cdots,p_6)^T,
\] 
where $\{p_1,p_2,\cdots p_6\}$ are shown in \eqref{eqn:p}. \vs
In this paper, we propose the methodology of equivariant gradient degree---first developed by K. Geba in \cite{gkeba1997degree}---to detect possible periodic solutions to \eqref{eqn01} near the orbit of this stationary point. The equivariant gradient degree extends the classical Brouwer and Leray–Schauder degree theories to the setting of gradient maps with symmetry. Below, we briefly outline the main idea.\vs
\noi Given a compact Lie group $G$, a $G$-invariant map $\varphi_{\lambda}$ and a neighborhood $\mathscr U$ of $G$-orbit of equilibrium $v^0$, the $G$-equivariant gradient degree 
$\nabla_{G}\text{\textrm{-deg}}\Big(\nabla \varphi
,\mathscr U\Big)$ is a well-defined element of the Euler ring $U(G)=\mathbb Z[\Phi(G)].$ Here, 
$\Phi(G)$ denotes the set of conjugacy classes $(H)$ of closed subgroups $H\leq G$. Thus, $\mathbb Z[\Phi(G)]$ is the free $\mathbb Z$-module generated by these classes. Then the degree can be written as
\[
\nabla_{G}\text{\textrm{-deg}}\Big(\nabla \varphi
,\mathscr U\Big)=n_1(H_1)+n_2(H_2)+\cdots n_k(H_k),\;\; n_j\in \mathbb Z
\]
where $(H_j)$ represents an orbit type  in $\mathscr U.$ The corresponding equivariant topological invariant at a critical value $\lambda_o$ is defined by
\[
\omega_G(\lambda_o):=\nabla_{G}\text{\textrm{-deg}}\Big(\nabla \varphi_{\lambda_-}
,\mathscr U\Big)-\nabla_{G}\text{\textrm{-deg}}\Big(\nabla \varphi_{\lambda_+}
,\mathscr U\Big),
\]
and takes the form
\[
\omega_G(\lambda_o)=r_1(K_1)+r_2(K_2)+\cdots +r_m(K_m),\;\;\; r_j\in \mathbb Z
\]
This invariant provides full classifications of the periodic solutions bifurcating from the equilibrium when $\lambda$ crosses $\lambda_o.$ For each nonzero coefficient $r_j,$ a global family of periodic vibrational modes emerges, with symmetry of at least $K_j$. It is worthy to note that our method provides alternative to other tools such as equivariant singularity, bush theory with density functional theory (\cite{golubitsky,chechin2003existence,chechin2015nonlinear}) for studying molecule dynamics, and it is among  many of other degrees such as primary degree, twisted degree, etc., which are all closely related to one another. See\cite{balanov2025degree, balanov2006applied, hooton2017noninvasive, liu2023existence, eze2022subharmonic, balanov2010periodic, ize2003equivariant,ruan2008applications, rybicki1997applications, dabkowski2017twisted,garcia2011global} for details of those degrees and some of the applications. One is also refered to Appendix \eqref{app:D} for some essential properties of equivariant gradient degree.\vs

Our main result, obtained through the application of the equivariant gradient degree, is stated in Theorem \eqref{thm:main}. Due to the octahedral configuration, the system exhibits eigenvalue multiplicities associated with resonance phenomena. Specifically, among the 15 nonzero eigenvalues, $\alpha^2_0$ has multiplicity 1, $\alpha^2_4$ has multiplicity 2, and each of $\alpha^2 _7,\alpha^2_{7^*},\alpha^2_8,\alpha^2_9$ has multiplicity 3 (see Table
\ref{tbl:eigen}). Classical result, such as the Lyapunov center theorem (cf. \cite{mawhin2013critical}), ensures the existence of periodic solutions near simple eigenvalues like  $\alpha_0$. Weinstein-Moser (cf. \cite{weinstein1973normal}) guarantees the existence of at least 15 periodic orbits on each sufficiently small fixed energy level. Our method shows that, under the assumption of isotypic nonresonance, i.e. no eigenvalue appears in more than one isotypic component, the newtonian system \eqref{eqn01} undergoes periodic solutions corresponding to
at least 16 distinct symmetry types of maximal orbit kinds, at critical numbers $\lambda_{j,l}$ with $j=0,4,7,7^*,8,9; \; l=1$. \textit{Among these symmetries, three correspond to spatially symmetric standing waves, eight exhibit anti-phase oscillations among individual atoms, and the remaining five are rotating waves.} \vs%The following presents results from the first several invariants for 1 folding. 
Article \cite{chechin2003existence} provides a thorough theoretical analysis of the normal modes of general octahedral molecules based on bush theory. Building on this, article \cite{chechin2015nonlinear} focuses specifically on the $SF_6$ molecule, verifying the conclusions of \cite{chechin2003existence} through density functional theory (to capture polarization effects). Both studies adopt the Lennard-Jones potential as the force field. In our work, we further extend the model by introducing additional interactions (see \eqref{eq:U1}), and by accounting for both spatial and temporal (reversion and phase shift) symmetries.  The vibrational modes identified in both studies are reproduced in our results, along with additional modes detected.\vs

The paper is organized as follows: In Section \ref{sec:2}, the mathematical model for octahedral molecule is introduced. Subsection \ref{sec：2.1} discusses the force field and its properties. Subsection \ref{sec:2.2} analyzes the symmetry group $S_4\times \mathbb Z_2\times O(3).$ Subsection \ref{sec:2.3} computes equilibrium $v^o.$ Isotypic decomposition and spectrum of Hessian are discussed in Subsection \ref{sec:2.4}, \ref{sec:2.5}. In Section \ref{sec:3}, the framework for equivariant bifurcation is established. \ref{sec:3.1}-\ref{sec:3.3} cover the functional space reformulation and the computation of critical numbers. The gradient degree is computed in Subsection \ref{sec:3.4}, and the main result, Theorem\eqref{thm:main}, along with its proof, is presented in Subsection \ref{sec:3.5}.  In Section \ref{sec:4}, we describe the details of the symmetries derived in Section 3, and also provide the animations of all these symmetries. The details of $S_4\times \mathbb Z_2$ subgroups, the properties of Euler ring and  equivariant gradient degree are shown in Appendix \ref{app:A},\ref{app:C},\ref{app:D}, respectively. We also track the computation of Hessian and its eigenspace in Appendix \ref{app:pre}.  \vs 
\noi{\bf Acknowledgment:} I am deeply grateful to Dr. Carlos Garc\'ia-Azpeitia for suggesting this topic and for his guidance and expertise in this area. I also truly appreciate Dr. Wiesław Krawcewicz for his support throughout the project and valuable insights. Many thanks as well to Dr. Dmitry Rachinskiy for his encouragement and support along the way.

\section{Mathematical framework}\label{sec:2}
\subsection{Force field.} \label{sec：2.1}
Let $V=\mathbb R^{18}$ and consider 6 identical ligand atoms $v:=(v_1,v_2,v_3,v_4,v_5,v_6)^T\in V$ %Put
%\[
% V':=\{(v_1,v_2,v_3,v_4,v_5,v_6)^T\in \mathbb R^{18}:v_1+v_2+v_3+v_4+v_5+v_6=0\}
%\]
and
\[
\Omega_o:=\{v\in V:v_{k}
\not =v_{j}\text{ for }k\not =j\}.
\]
The overall interactions of ligand-ligand, ligand-central atoms are represented by Newtonian equation:
\[
\ddot{v}=-\nabla U(v),\;\;v\in \Omega_o
\]
where potential $U:\Omega_o\to \mathbb R$ is given by
\[
U(v):=\sum_{1\leq j<k\leq6}U_1(|v_{j}-v_{k}|^2)+\sum_{1\leq j<k\leq6}U_2(|v_{j}-\mathcal O|^2),
\]
and $U_1, U_2$ take the form in equation \eqref{eq:U1}.
Notice that since both $U_1$, $U_2$ belong to $C^2(\mathbb R^+) $ and
satisfy 
\begin{equation}\label{eqn:coercive}
\lim_{s^{+}\rightarrow0}U_1(s)+U_2(s)=\infty,\quad\lim_{s\rightarrow\infty}%
U_1(s)+U_2(s)=\infty.
\end{equation}
The combined potential $U$ diverges both as particles approach each other and as they move infinitely
far apart. This physically prevents particle collapse or dispersion, and ensures that $U$ is well-defined. Moreover, $U$ is also of class $C^2$ on the domain where particle distances remain
positive.
\subsection{Octahedral Symmetries.} \label{sec:2.2}
The space $V$ is a representation of the group 
\[
G':=S_4\times \mathbb Z_2\times O(3),
\]
which describes the full spatial symmetry of the octahedral molecule. 
Note that $S_4\times \mathbb Z_2$ can be identified with a subgroup $\mathbf O\leq S_6$ through the following correspondence between its elements and those of $S_6$:
\begin{align}\label{eq:s6s4p}
((24),1)&\leftrightarrow (14)(23)(56),\;\;
((12)(34),1)\leftrightarrow (13)(56),\;\;\;((1),-1)\leftrightarrow (13)(24)(56),\\
((132),1)&\leftrightarrow (146)(253),\;\quad
((1234),1)\leftrightarrow (1234).\;\nonumber
\end{align}
The action of $G'$ on $\mathbb R^{18}$ takes the form
\begin{equation}\label{eq:action}
(\sigma, M)(v_1,v_2,v_3,v_4,v_5,v_6)^T=(
Mv_{\sigma(1)}, Mv_{\sigma(2)}, Mv_{\sigma(3)}, Mv_{\sigma(4)}, Mv_{\sigma(5)}, M v_{\sigma(6)})^T,
\end{equation}
where $\sigma\in S_6, M\in O(3).$\vs
\noi We now regard $S_4\times \mathbb Z_2$ as a subgroup of $O(3).$ Notice that it naturally corresponds to the full symmetry of a regular octahedron $\textbf{P}\subset \mathbb R^3.$ Indeed, let the six vertices of $\textbf{P}=\{p_1,p_2,p_3,p_4,p_5,p_6\}\;$ be given by
\begin{equation}\label{eqn:p}
p_1=-p_3=(1,0,0)^T,\quad
p_2=-p_4=(0,1,0)^T,\quad
p_5=-p_6=(0,0,1)^T
\end{equation}
as illustrated in Figure \ref{fig:oct}. Then permutation $\sigma$ of 
$S_4\times \mathbb Z_2$ can be identified with matrix $M_{\sigma}\in O(3),$ where $M_{\sigma}$ denotes the octahedral group satisfying $M_{\sigma}\textbf{P}=\textbf{P},$ via the correspondence
%For any element $\sigma\in S_4\times \mathbb Z_2,$ one can explicitly define an isomorphism $M_{\sigma}: S_4\times \mathbb Z_2\to O(3),$ and
\[
M_{\sigma}p_j=p_{\sigma(j)},\qquad j\in \{1,2,\cdots,6\}.
\]

\noi To be more precise,  the generators of $S_4\times \mathbb Z_2$, namely, $
((132),1), ((1234),1),((1),-1)$, 
admits the following identification with orthogonal matrices in $O(3)$
\begin{equation}\label{enq:identification}
M_{((132),1)}=\left[
\begin{array}
[c]{ccc}%
0 & 0 & -1\\
-1 & 0 & 0\\
0 & 1 & 0%
\end{array}
\right],\quad
M_{((1234),1)}
=\left[
\begin{array}
[c]{ccc}%
0 & -1 & 0\\
1 & 0 & 0\\
0 & 0 & 1
\end{array}
\right],\;\;
%M_{((234),1)}&=\left%[
%\begin{array}
%[c]{ccc}%
%-1 & 0 & 0\\
%0 & 1 & 0\\
%0 & 0 & -1
%\end{array}
%\right],\quad
M_{((1),-1)}=\left[
\begin{array}
[c]{ccc}%
-1 & 0 & 0\\
0 & -1 & 0\\
0 & 0 & -1
\end{array}
\right].
\end{equation}
\subsection{Octahedral Equilibrium.}\label{sec:2.3} Notice that potential $U$ is $G'$ invariant. The following is based on the application of Palais’s Principle of Symmetric Criticality to locate critical point. For notational convenience, we hereinafter denote $S_4\times \mathbb Z_2$ by $S_4^p.$\vs
\noi Consider point $p^o=(p_1,p_2,p_3,p_4,p_5,p_6)^T \in \Omega_o.$ Its isotropy group $G'_{p^o}$ is given by
\[
\tilde{S}_4^{p}:=\left\{(\sigma,M_{\sigma})\in S_4^p\times O(3): \sigma \in S_4^p\right\},
\]
where $S_4^p$ is regarded as a subgroup of $S_6$ embedded in $O(3)$ through the identification described above. Moreover, a direct computation shows that the fixed-point subspace $V^{\tilde{S}_4^p}$ is one dimensional and given by
\[
V^{\tilde{S}_4^p}=\text{span}_{\mathbb R}\{(p_1,p_2,p_3,p_4,p_5,p_6)^T\}.
\]
%Indeed, let $v=(v_1,v_2,v_3,v_4,v_5,v_6)^T$ where each $v_j$ is written as $v_j=(v_j^1,v_j^2,v_j^3)^T\; (j=1,\cdots,6).$ Under the group action \eqref{eq:action} %of $S_4^p$
%and the matrix identification \eqref{enq:identification}, 
So any vector $v\in V^{\tilde{S}_4^p}$ must take the form  $v=r p^o,\; r\in \mathbb R.$ Now, define
\[
\varphi (r):=U(rp^o)=12U_1(\sqrt{2}r)+3U_1(2r)+6U_2(r),\quad (r>0)
\]
Since $\varphi(r)$ is the restriction of $U$ to the fixed-point subspace $V^{\tilde{S}_4^p}\cap \Omega_o,$ and since $\varphi(r)$ satisfies \eqref{eqn:coercive}, there must exist $r^o$ such that $r^o$ is a minimizer. By Symmetric Criticality Principle, this minimizer implies that $r^op^o$ is a critical point of $U.$ i.e. the equilibrium $v^o$ of system \eqref{eqn01} takes the form
\begin{equation}\label{eq:critical}
  v^o=r^o p^o.  
\end{equation}
\subsection{Isotypic Decomposition.} \label{sec:2.4}
Given the $G'$ orbit of the equilibrium $v^o,$ the tangent space to the orbit at $v^o$, denoted $T_{v^o}G'(v^o)$, can be obtained from the Lie algebra $\frak{so}(3)$ of $SO(3)$ (notice, $S_4$ and $\mathbb Z_2$ are discrete). 
In particular, let $J_i, i=1,2,3$ denote the three infinitesimal generators of rotations in  $\mathbb R^3$ (i.e. $e^{\Phi J_1}, e^{\theta J_2}, e^{\Psi J_3}, $ where $\Phi,\theta,\Psi$ are Euler angles). Then these generators are explicitly given by%
\[
J_{1}:=\left[
\begin{array}
[c]{ccc}%
0 & 0 & 0\\
0 & 0 & -1\\
0 & 1 & 0
\end{array}
\right]  ,\quad J_{2}:=\left[
\begin{array}
[c]{ccc}%
0 & 0 & 1\\
0 & 0 & 0\\
-1 & 0 & 0
\end{array}
\right]  ,\quad J_{3}:=\left[
\begin{array}
[c]{ccc}%
0 & -1 & 0\\
1 & 0 & 0\\
0 & 0 & 0
\end{array}
\right],
\]
which form a basis of $\frak {so}(3).$ Therefore, the tangent space and slice at $v^o$ is given by
\[
T_{v^{o}}G'(v^{o})=\text{span}\{(J_{i}p_{1},J_{i}p_{2},J_{i}p_{3},J_{i}p_{4},J_{i}p_{5},J_{i}p_{6})^{T}\in V:i=1,2,3\},
\]
and
\[
S_o:=\{v\in V: v\cdot T_{v^o}G'(v^o)=0\}.
\]\vs
Since $v^o$ has isotropy $\tilde{S}_4^p$, so $S_o$ is an $\tilde{S}_4^p$ orthogonal representation. The purpose of the following is to obtain the $\tilde{S}_4^p$ isotypic decomposition of the slice $S_o$. We first study the character table of $S_4^p,$ as shown in Table \eqref{tbl:characters_s4}, where $\chi_j,j=0,\cdots,9$ denotes the characters of all $S_4^p$ irreducible representation $\mathcal{W}_j$ and $\mathcal{\chi}_{V}$ is the character of $V.$

\begin{table}[H] 
\centering
\resizebox{\textwidth}{!}{
\begin{tabular}{|c|c|c|c|c|c|c|c|c|c|c|c|}
\hline
Rep. & Character & ((1),1)& ((1),-1) & ((12),1) &((12),-1) & ((12)(34),1) &((12)(34),-1) & ((123),1) &((123),-1) & ((1234),1)&((1234),-1)\\\hline
$\mathcal{W}_{0}$ & $\chi_{0}$ & 1 &1 & 1 & 1&  1 & 1& 1 &1 & 1 &1\\
$\mathcal{W}_{1}$ & $\chi_{1}$ & 1 &-1 & 1 & -1&  1 & -1& 1 &-1 & 1 &-1\\
$\mathcal{W}_{2}$ & $\chi_{2}$ & 1 &1 & -1 &-1 & 1 &1 & 1 &1 & -1& -1\\
$\mathcal{W}_{3}$ & $\chi_{3}$ & 1 &-1 & -1 &1 & 1 &-1 & 1 &-1 & -1& 1\\
$\mathcal{W}_{4}$ & $\chi_{4}$ & 2 &2 & 0 &0 & 2 &2 & -1 & -1 & 0& 0\\
$\mathcal{W}_{5}$ & $\chi_{5}$ & 2 &-2 & 0 &0 & 2 &-2 & -1 & 1 & 0& 0\\
$\mathcal{W}_{6}$ & $\chi_{6}$ & 3 &3 & -1 & -1& -1 &-1 & 0 &0 & 1&1\\
$\mathcal{W}_{7}$ & $\chi_{7}$ & 3 &-3 & -1 & 1& -1 &1 & 0 &0 & 1&-1\\
$\mathcal{W}_{8}$ & $\chi_{8}$ & 3 & 3& 1 &1& -1 &-1& 0 &0 & -1&-1\\
$\mathcal{W}_{9}$ & $\chi_{9}$ & 3 & -3& 1 &-1& -1 &1& 0 &0 & -1&1\\\hline
$V=\mathbb{R}^{18}$ & $\chi_{V}$ & $18$ &0  &0  &2  &-2 & 4& 0 &0& 2& 0
\\\hline
\end{tabular}
}
\caption{Character Table of $S_4^p$}
\label{tbl:characters_s4}
\end{table}
\noi By comparing $\chi_{V}$ and characters in Table \eqref{tbl:characters_s4}, one has
\[
V=\mathcal{W}_{0}\oplus \mathcal W_4\oplus \mathcal W_6\oplus 2\mathcal W_7\oplus\mathcal{W}_{8}\oplus \mathcal W_9.
\]
%Notice that $\mathbb R^{18}=V'\oplus \tilde V,$
%where 
%\[
%\tilde V=\{(v,v,v,v,v,v)^T\in %\mathbb R^{18}:v\in \mathbb R^3\}.
%\]
%\textcolor{purple}{
%It is obvious that $\tilde V$ is equivalent to $S_4^*$ irreducible representation $\mathcal W_9.$ Indeed, $\tilde V$ is $S_4^*$-invariant and with degree of freedom 3, which implies that the corresponding representation on 
%$\tilde V$ is potentially isomorphic to one of the irreducible 3-dimensional representations of  $S_4^*$, namely, $\mathcal W_6, \mathcal W_7,\mathcal W_8, \mathcal W_9$. The precise identification is determined by the action of $M_{\sigma}\in O(3),\;\sigma\in S_4\times \mathbb Z_2$ defined in \eqref{eq:identification}, and by comparing the traces of $M_{((234),1)}, M_{((1234),1)},M_{((1),-1)}$ with those in  table \eqref{eq:table}.} \vs \noi Thus, one has
%\[
%V'=\mathcal W_0\oplus \mathcal W_4\oplus \mathcal W_6\oplus 2 \mathcal W_7\oplus \mathcal W_8,
%\]
Moreover, since
\[
\dim T_{v^{o}}G'(v^{o})=\dim G'(v^o)=\dim G'-\dim G'_{v^o}=3,
\]
and the eigenvalue $\alpha_6=0$ has multiplicity exactly 3 (see Table \ref{tbl:eigen}), we conclude explicitly that
\[
T_{v^o}G'(v^o)=\mathcal W_6.
\]
%\textcolor{red}{Moreover, we notice that
%\[
%\dim T_{v^{o}}G'(v^{o})=\dim G'(v^o)=\dim G'-\dim G'_{v^o}=3.
%\]
%Since $U$ is $G'$ invariant and
%it follows that $\nabla^2 U(v^o)$ vanishes along the tangent direction of the orbit, i.e. $T_{v^o}G'(v^o)\subset \ker (\nabla^2 U(v^o)).$
%So given from Table \eqref{tbl:eigen} that eigenvalue $\alpha_6=0$ has multiplicity exactly $3,$ we conclude
%\[
%T_{v^o}G'(v^o)=\mathcal W_6.
%\]}
\begin{comment}
Indeed, for 
 any $v\in T_{v_{o}}G'(v_{o}),$ one can easily verify that $P_3 v=P_4v=0.$ %On the other hand, let 
% \[\mathcal K\in \{((12),1), ((12)(34),1), ((1234),1)\},\]
%and denote by $\mathcal C_{\mathcal K}$ the conjugacy class of $\mathcal K.$ By applying Schur Lemma and concept of Center of Group Algebra, one has for every $v\in V_j,$ there exists scalar $c_{\mathcal K}$ such that 
%\[
%\sum_{g\in \mathcal C_{\mathcal K}}g\cdot v=c_{\mathcal K}v,
%\]
%where $c_{\mathcal K}=\frac{|\mathcal C_{\mathcal K}|\chi_j(g)}{\dim V_j}.$
%Then by direct computation, for basis $v$ of $T_{v_{o}}G'(v_{o}),$ one has
%\[
%P_{\mathcal W_3^-}v=\frac{3}{48}\Big(6v-2\sum_{g\in\mathcal C_{((12),1)}}g\cdot v-2\sum_{g\in\mathcal C_{((12)(34),1)}}g\cdot v+2\sum_{g\in\mathcal C_{((1234),1)}}g\cdot v\Big),
%\]
%Then we can derive
%\[
%P_{\mathcal W_3^-}v=\frac{1}{8}\Big(3v+2v+v+2v\Big)=v.
%\]
\end{comment}
\noi Thus slice $S_o$ has the following isotypic decomposition
\[
S_o=V_0\oplus V_4\oplus V_7\oplus V_7^* \oplus  V_8\oplus V_9,\quad\quad  
\] 
where $V_j=\mathcal W_j, (j=0,4,7,7^*,8,9)$ and $V_7, V_7^*$ are equivalent.\vs
\subsection{Computation of Hessian $\nabla^{2}U(v^o)$ and Its Spectrum.}\label{sec:2.5}
Notice that the explicit form of gradient $\nabla U(v)$ is given by
\[
\nabla U(v)=2\left[
\begin{array}
[c]{c}%
\sum_{k\not =1}U_1^{\prime}(|v_{1}-v_{k}|^{2})(v_{1}-v_{k})+U_2^{\prime}(|v_1-\mathcal O|^2)(v_1-\mathcal O)\\
\sum_{k\not =2}U^{\prime}(|v_{2}-v_{k}|^{2})(v_{2}-v_{k})+U_2^{\prime}(|v_2-\mathcal O|^2)(v_2-\mathcal O)\\
\vdots\\
\sum_{k\not =6}U^{\prime}(|v_{6}-v_{k}|^{2})(v_{6}-u_{k})+U_2^{\prime}(|v_6-\mathcal O|^2)(v_6-\mathcal O).
\end{array}
\right]
\]
In such a case, we say that $v^o=r^op^o$ is a critical point if $r^o$ satisfies the condition 
\begin{equation}\label{eq:CT}
4U_1'(2{r^o}^2)+2U_1'(4{r^o}^2)=-U_2'({r^o}^2).
\end{equation}
Given $v=(x,y,z)^T\in \mathbb R^3,$ define \[\mathfrak m_v:=vv^T=
\begin{bmatrix}
    x^2&xy&xz\\xy&y^2&yz\\xz&yx&z^2
\end{bmatrix}.\] For $j\neq k,$ set $\mathfrak m_{jk}=\mathfrak m_{p_j-p_k},$ it is obvious that $\mathfrak m_{jk}=\mathfrak m_{kj}.$ For each $j\in \{1,\cdots,6\},$ define 
\[
\Lambda_j:=\{k\in \{1\cdots 6\}: k\neq j,\;\; p_k \sim p_j\},
\]
where $p_k\sim p_j$ implies that $p_k$ and $p_j$ are adjacent (see Figure \ref{fig:oct}).\vs
\noi For simplicity of notation, hereafter we define
\[
a:= U_1^{''}(2{r^o}^2),\; b := U_1^{''}(4{r^o}^2),\; c := U_2^{''}({r^o}^2), d:= U_1^{'}(4{r^o}^2),\; e := U_1^{'}(2{r^o}^2),
\]
and denote by $P_{jk}$ the entries of Hessian, which is explicitly given by 
\[
P_{jk}=\frac{\partial^2 U(v^o)}{\partial p_k\partial p_j}.
\] 
Then a direct computation shows that the Hessian at $v^o$ takes the form
\[
M:=\nabla^2 U(v^o)=
\begin{bmatrix}
 P_{11}&\cdots& P_{16}  \\
\vdots &&\vdots  \\ 
 P_{61} & \cdots& P_{66}
\end{bmatrix},
\]
where
\begin{equation}\label{eqn: Pjk}
P_{jk}=
\begin{cases}
    -2a\mathfrak m_{jk}-e\id,&\quad k\in \Lambda_j\\
    -2b\mathfrak m_{jk}-d\id,&\quad k\notin \Lambda_j\cup \{j\}\\
    2a\sum_{k\in \Lambda_j}\mathfrak m_{jk}+2b\sum_{k\notin \Lambda_j\cup\{j\}}\mathfrak m_{jk}+2c\mathfrak m_{j0}-d\id,&\quad k=j.\\
\end{cases}
\end{equation}
 Define 
\[
\rho:=\sqrt{36a^2+4ac+c^2+36ae+2ce+9e^2},
\]
the eigensystem of $\nabla^2 U(v^o)$ is presented in Table \eqref{tbl:eigen}. For completeness, explicit expressions for each entry $P_{jk}$ %\textcolor{red}{and detailed descriptions of corresponding eigenspaces} 
are provided in Appendix \ref{app:pre}. 

\begin{table}[h!]
\begin{tabular}{|l|c|c|}
\hline
Eigenvalue $\alpha^2_j$ & Eigenspace & Multiplicity \\ \hline
 $\alpha^2_0=2(8a+8b+c)$ & $\mathcal W_0$& 1 \\ \hline
$\alpha^2_4=2(2a+8b+c)$&$\mathcal W_4$ & 2 \\ \hline
$\alpha^2_6=0$&$\mathcal W_6$& 3 \\ \hline
$\alpha^2_7=6a+c-2d-e-\rho$&$\mathcal W_7$& 3 \\ \hline
$\alpha^2_{7^*}= 6a+c-2d-e+\rho$&$\mathcal W_7^*$ &3 \\ \hline
$\alpha^2_8= 8a$ & $\mathcal W_8$ &3 \\ \hline
$\alpha^2_9= 2(2a-d+e)$ &$\mathcal W_9$& 3 \\ \hline
\end{tabular}
\caption{Eigenvalue and Eigenspace of Hessian $\nabla^2U(v^o)$}
\label{tbl:eigen}
\end{table}

\section{Equivariant Bifurcation}\label{sec:3}
In what follows, our goal is to identify nontrivial $T$-periodic solutions to system \eqref{eqn01} that bifurcate from the orbit $G'(v^o)$ of equilibrium. We first normalize the period via the substitution 
\[
u(t)=v(\lambda t),
\]
where $\lambda=T/2\pi$ and $\lambda^{-1}$ denotes the frequency. Then system \eqref{eqn01} transforms into the following equivalent form
\begin{equation}\label{eqn02}
\begin{cases}
    \ddot u=-\lambda^2 \nabla U(u)\\
    u(0)=u(2\pi),\;\; \dot u(0)=\dot u(2\pi)
\end{cases}
\end{equation}
\subsection{Functional Space Reformulation.} \label{sec:3.1}Consider Sobolev space of $2\pi$-periodic functions 
\[
H^1_{2\pi}(\mathbb R;V):=\Big\{v:\mathbb R\to V:v(0)=v(2\pi), v|_{[0,2\pi]}\in H^1([0,2\pi];V)
\Big\},
\]
equipped with the inner product
\[
\langle u,v \rangle=\int_0^{2\pi}\Big(u(t)\cdot v(t)+\dot{u}(t)\cdot \dot{v}(t)\Big)dt.
\]
Let $O(2)=SO(2)\cup \kappa SO(2)$ denote the group of all $2\times 2$ orthogonal matrices,where 
\[
SO(2)=\Big\{\begin{bmatrix}
    \cos \tau& -\sin\tau\\
    \sin \tau & \cos \tau
\end{bmatrix}\Big\},\quad \kappa=
\begin{bmatrix}
    1&0\\0&-1
\end{bmatrix},
\]
Each element of $SO(2)$ corresponds naturally to a rotation $e^{i\tau}\in S^1$, and the element $\kappa$ represents reflection, satisfying the relation $\kappa e^{i\tau}=e^{-i\tau}\kappa$.\vs

It is clear that $H^1_{2\pi}(\mathbb R;V)$ is orthogonal Hilbert representation of $G$, where 
\[
G=O(2)\times G',\;\; G'=S_4\times \mathbb Z_2\times O(3).
\]
Indeed, we have, for each $v\in H^1_{2\pi}(\mathbb R,V), $ 
\begin{align}\label{eqn:action}
    (e^{i\tau},\sigma, M)v(t)&=(\sigma, M)v(t+\tau)\\
    (\kappa e^{i\tau},\sigma, M)v(t)&=(\sigma, M)v(-t+\tau).\notag
\end{align}
Moreover, one can derive the identification between the $2\pi$-periodic function $x:\mathbb{R}\rightarrow
V$ and $\widetilde{x}:S^{1}\rightarrow V$ through the following commuting diagram:
\[
\begin{tikzcd}[column sep=4em]
\mathbb{R} \arrow[dr, "x"'] \arrow[r, "\mathfrak{e}(\tau)=e^{i\tau}"] & S^1 \arrow[d, "\tilde{x}"] \\
 & V
\end{tikzcd}.
\]
Therefore, $H_{2\pi
}^{1}(\mathbb{R}, V) \simeq H^{1}(S^{1}, V)$. \vs Define
$
\Omega:=\{v\in H^1(S^1;V):v(t)\in \Omega_o\}.
$
We then introduce the energy functional $
J:\mathbb R\times \Omega\to\mathbb R$ given by
\begin{equation}
J(\lambda,v):=\int_{0}^{2\pi}\left[  \frac{1}{2}|\dot{v}(t)|^{2}-\lambda
^{2}U(v(t))\right]  dt. \label{eq:var-1}%
\end{equation}
Thus, system \eqref{eqn02} can now be expressed as the variational equation
\begin{equation}\label{eqn03}
\nabla_{v}J(\lambda,v)=0,\;\; (\lambda, v)\in \mathbb R\times \Omega.
\end{equation}\vs
By the Riesz representation theorem for Hilbert spaces, the space $H^1(S^1;V)$ can be identified with its dual $H^1(S^1;V)^{*}$ via the inverse Riesz map $L : H^1(S^1;V) \to H^1(S^1;V)^{*}$, given by
\[
Lv=-\ddot v+v.
\]
Consequently, the gradient operator $\nabla_{v}J(\lambda,v):H^1(S^1;V)\to H^1(S^1;V)$ can be uniquely defined by
\[
 \nabla_vJ(\lambda, v):=L^{-1} \delta_vJ(\lambda,v)=L^{-1}(-\ddot u-\lambda^2 U(v)),
\]
where  $\delta_v J(\lambda, v) :H^1(S^1;V)\to H^1(S^1;V)^*.$ Therefore, solving \eqref{eqn03} is equivalent to solving system $\nabla_v J(\lambda,v)=0.$  One can  verify that for $v\in H^1(S^1;V),$
\[
 \nabla_v J(\lambda, v)=v-L^{-1}(\lambda^2 \nabla U(v)+v),
\]
where $L^{-1}(\lambda^2 \nabla U(v)+v)$ is a compact operator. Moreover, the linearization $\nabla_v^2 J(\lambda,v)$ at $v^o$ is given by
\begin{equation}\label{eq:nabla2}
\nabla_v^2J(\lambda, v^o)=\id-j\circ L^{-1}(\lambda^{2}\nabla^2 U(v^o)+\id),
\end{equation}

%Now put $L:H^{2}(S^{1}; V')\rightarrow L^{2}%
%(S^{1};V')$ the operator given by 
%\[
%Lv:=-\ddot{v}+v,
%\]
%then clearly, operator $L^{-1}$ exists and is bounded 
%(use Fourier Analysis to explicitly prove)
%. 
%Let $j:H^{2}(S^{1}; V')\hookrightarrow H^{1}(S^{1}; V')$ be the natural embedding,  then by Arzelà–Ascoli theorem, $j$ is compact. Then one has the following completely continuous field $\nabla_{v}J:\mathbb R\times H^1(S^1;V')\to H^1(S^1;V')$ given by
%\begin{equation}
%\nabla_{v}J(\lambda,v)=v-j\circ L^{-1}(\lambda^{2}\nabla U(v)+v),\qquad v\in H^1(S^1;V'),
%\label{eq:gradJ}%
%\end{equation}
%and accordingly, the linearization $\nabla^2_vJ(\lambda, v)$ at $v^o$ is given by
%\[
%\nabla_v^2J(\lambda, v^o)=\id-j\circ L^{-1}(\lambda^{2}\nabla^2 U(v^o)+\id),
%\]
\subsection{Application of Slice Criticality Principle.} \label{sec:3.2} Consider the equilibrium point $v^o$ discussed previously, which is clearly also a critical point of $J$. Recall that our primary goal is to identify non-stationary $2\pi$-periodic solutions bifurcating from the equilibrium $v^o,$ that is,  non-constant solutions of system \eqref{eqn02}. The analysis following is founded on the Slice Criticality Principle (see Theorem \eqref{thm:SCP} for details), a tool for the computation of equivariant gradient degree of $\nabla J_{\lambda}$. \vs
We consider the orbit $G(v^o)$ within the space $H^1(S^1;V),$ and let  $\bf{S}_o$ denote the slice to $G(v^o)$ at $v^o$ in this space.
We then define the restriction
\[
\mathbf J(\lambda, v):=J(\lambda, v)|_{\mathbb{R}\times (\mathbf S_o\cap \Omega)}.
\]
Clearly, $\mathbf  J$ remains invariant under the isotropy group $G_{v^o}$. Now one can utilize this restriction to apply Slice Criticality Principle for computation of equivariant gradient degree of $\nabla J_{\lambda}$ on a small neiborhood $\mathscr U$ of $G(v^o)$.\vs
Consider operator 
\begin{equation}\label{eqn:linearization}
\mathscr A(\lambda):=\nabla^2_v J(\lambda, v^o)|_{\mathbf S_o}:\bf S_o\to \bf S_o.
\end{equation}

Since $
\mathscr A(\lambda)=
\nabla_{v}^{2}\mathbf J(\lambda,v^{o}),$ one can verify that $G(v^{o})$ is a finite dimensional isolated orbit of
critical points of $J$ whenever $\mathscr A(\lambda)$ is an isomorphism.
\begin{comment}Indeed, given that the Hessian at $v^o$, when restricted to the slice $\bf S_o$, is nondegenerate, the Implicit Function Theorem guarantees the local uniqueness and isolation of solution near $v^o$ within the slice. Due to the symmetry induced by the group action, this isolation property consequently extends to the entire orbit $G(v^o).$
\end{comment}\vs
\noi That is to say, if a point $(\lambda^{o},v^{o})$ is a bifurcation point for
\eqref{eqn03}, then $\mathscr A(\lambda^{o})$ cannot be an isomorphism. In
such a case we introduce the \textit{critical set} for trivial solution $v^o$
\[
\Lambda:=\{\lambda>0:\mathscr A(\lambda)\text{ is not an isomorphism}\}.
\]
\subsection{Slice $\mathbf S_o$ and Critical Numbers.}\label{sec:3.3} Let $S^1$ acts on $H^1(S^1;V)$ by shifting argument (see \eqref{eqn:action}). It follows that $H^1(S^1;V)^{S^1}$ consists of constant functions and can be identified with $V.$ Thus, we obtain the decomposition
\[
H^{1}(S^{1}, V)= V\oplus W,\quad \text{ where }W:=V^{\perp}.
\]
Consequently, the slice $\mathbf S_o$ to $G(v^o)$ in $H^1(S^1;V)$ at $v^o$
has the explicit form
\[
\mathbf S_o=S_o\oplus W.
\]

%\begin{definition}
%We say that $\lambda^{o}\in\Lambda$ satisfies condition (A) if $\ \mathscr
%A(\lambda^{o})|_{S_{o}}:S_{o}\rightarrow S_{o}$ is an isomorphism.   
%\end{definition}
%\begin{theorem}
%\label{th:bif1} Assume
%$\lambda_{o}\in\Lambda$ satisfies condition (A) and there exists $\lambda_{-}<\lambda^{o}<\lambda_{+}$ such that $[\lambda
%_{-},\lambda_{+}]\cap\Lambda=\{\lambda^{o}\}$, i.e. $\lambda^{o}$ is isolated in $\Lambda$. Put $B_1(0)$ the open unit ball in $\mathbf S_o,$ and define the bifurcation invariant
%\[
%\omega_{G}(\lambda^{o}):=\nabla_{G_{v^{o}}}\text{\textrm{-deg}}\Big(\mathscr
%A(\lambda_{-}),B_{1}(0)\Big)-\nabla_{G_{v^{o}}}\text{\textrm{-deg}%
%}\Big(\mathscr A(\lambda_{+}),B_{1}(0)\Big),
%\]
%If
%\[
%\omega_{G}(\lambda^{o})=n_{1}(H_{1})+n_{2}(H_{2})+\dots+n_{k}(H_{k})
%\]
%is non-zero, i.e. $n_{j}\not =0$ for some $j=1,2,\dots,k$, then there exists a
%bifurcating branch of nontrivial solutions to \eqref{eq:bif1} from the orbit
%$\{\lambda^{o}\}\times %G(v^{o})$ with symmetries %at least $(H_{j})$.
%\end{theorem}

Consider the $S^{1}$-isotypic decomposition of $W$, given explicitly by
\[ W=\overline{\bigoplus_{l=1}^{\infty} W_{l}},\quad
W_{l}:=\{\cos(l\cdot)\mathfrak{a}+\sin(l\cdot)\mathfrak{b}:\mathfrak{a}%
,\,\mathfrak{b}\in V’\}
\]
In a standard way, each space $W_{l}$, $l=1,2,\dots$, can be naturally identified with the space $ {V}^{\mathbb{C}},$  on which the $S^{1}$-action is given by Fourier mode $l$. Specifically, we have
\[
W_{l}=\{e^{il\cdot}z:z\in {V’}^{\mathbb{C}}\}.
\]
Since the operator $\mathscr A(\lambda)$ is equivariant with respect to
\[
G_{v^{o}}=O(2)\times \tilde{S}^p_{4},
\]
it is thus $S^{1}$-equivariant as well. Hence we have $\mathscr A(\lambda)(
W_{l})\subset W_{l}$. By $\tilde{S}^p_{4}$-isotypic decomposition
of ${V}^{\mathbb{C}}$, we obtain a corresponding $G_{v^{o}}$-isotypic decomposition given by
\[
W_{l}=W_{0,l}\oplus  W_{4,l}\oplus W_{7,l}\oplus W_{7^*,l}\oplus W_{8,l}\oplus W_{9.l},\qquad
W_{j,l}=\mathcal{W}_{j,l}~.
\]
Moreover, we have
\[
\mathscr A(\lambda)|_{W_{j,l}}=\left(  1-\frac{\lambda^{2}\alpha^2_{j}+1}{l^{2}%
+1}\right)  \text{\textrm{Id\,}}~.
\]
Thus $\mathscr A(\lambda^{o})|_{W_{j,l}}=0$ if and only if $\lambda^{o}%
=\frac{l}{\alpha_{j}}$ for $l\in \mathbb N^+$ and $j=0,4,7,7^*,8,9$. We denote the critical numbers $\lambda\in\Lambda$ as
\[
\lambda_{j,l}=\frac{l}{\alpha_{j}}~,
\]
and the critical set $\Lambda$ associated with the equilibrium $v^{o}$ of the system
\eqref{eqn01} is described as
\[
\Lambda:=\left\{  \frac{l}{\alpha_{j}}:j=0,4,7,7^*,8,9,\quad l\in \mathbb N^+
\right\} .
\]
\noi For the case where equation \eqref{eq:U1} has the specific parameters $\sigma_1=\sigma_2=0.0618, \sigma_3=1,$ one can derive the equilibrium  $r^o=1.4128,$ along with the relation among different $\lambda_{j,l}:$
\begin{align}\label{sq:ordering}
\lambda_{0,1}&<\lambda_{7^*,1}<\lambda_{4,1}<\lambda_{7,1}<\lambda_{0,2}<\lambda_{8,1}
<\lambda_{7^*,2}<\lambda_{4,2}<\cdots
%<\lambda_{0,3}<\lambda_{7^*,3}<\lambda_{7,2}<\lambda_{4,3}
%<\lambda_{0,4}<\lambda_{8,2}<\lambda_{7^*,4}<\cdots\notag\\
\end{align}
For details of the first 30 pairs of $(\lambda_{j,l},\alpha_j),$ see Figure \eqref{fig:criticalnum}.
\begin{figure}[H]
 % \centering
\includegraphics[width=1.1\textwidth, keepaspectratio]{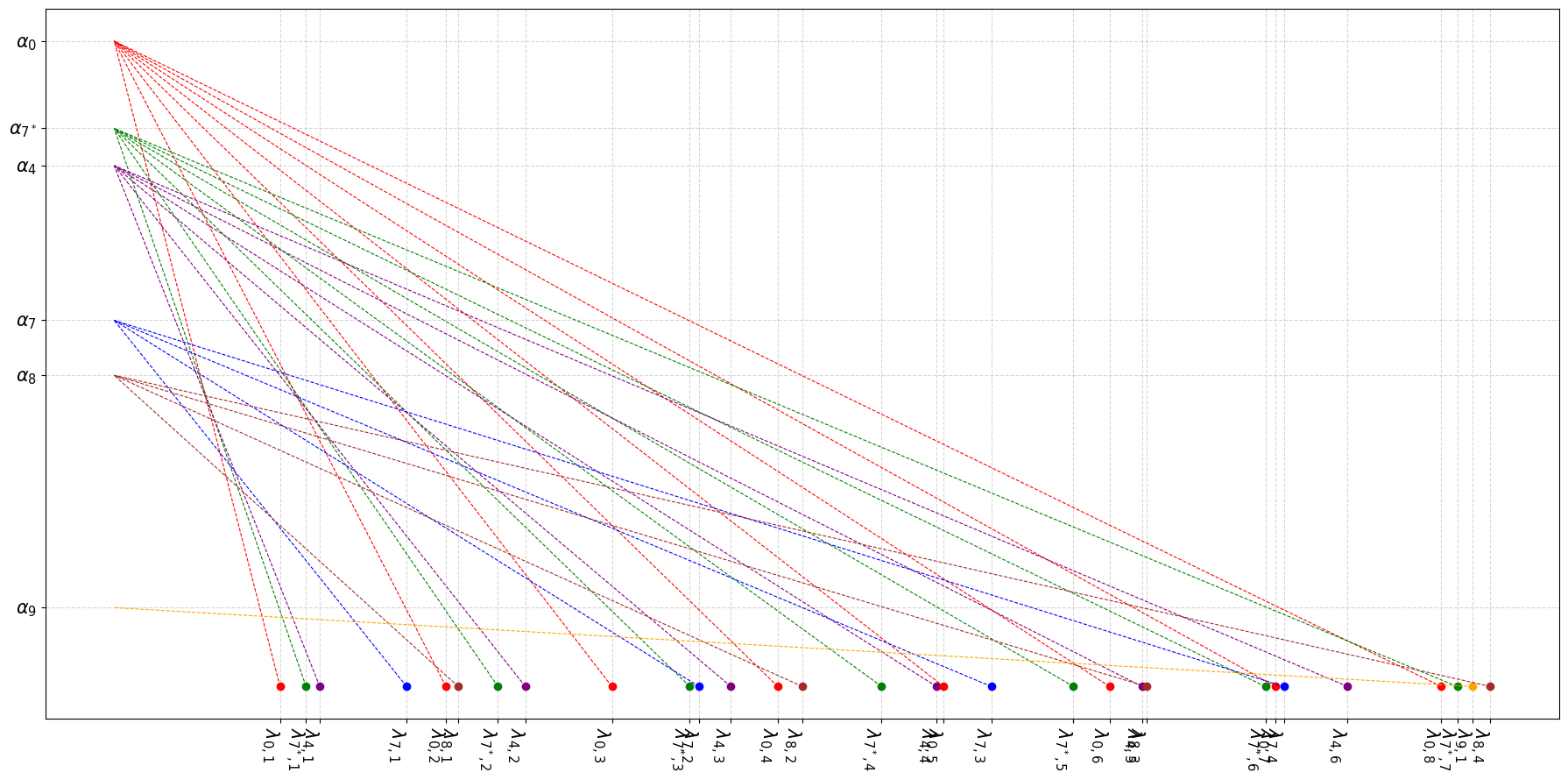}
  \caption{Critical Values $\lambda_{j,l}$}
  \label{fig:criticalnum}
\end{figure}

Notice, the nondimensionalized version of parameters $\sigma_1,\sigma_2,\sigma_3$ are based on actual numbers derived from numerical experiments in \cite{dellis}. Moreover, we need to point out that, in a general case of isotypic resonance, the critical numbers in $\Lambda$ cannot be uniquely identified with the indices $(j,l).$ However, one can observe that in our setting, $\alpha_j$ are pairwise distinct and given by \[
0.7867,0.5123,0.2532,0.5882,0.1829,0.01173.
\] Therefore, it is reasonable, hereinafter, to make such an assumption that all critical number $\lambda_{j,1}\;(j=0,4,7,7^*,8,9)$ are isotypic nonresonant. 
\subsection{Computation of the Gradient Degree.}\label{sec:3.4} 
Note that the bifurcation invariant $\omega_G(\lambda_{j,l})$ is expressed as an element in the Euler ring  $U(O(2)\times S_4^p).$ This framework allows for a complete classification of symmetry types with nonzero dimensional Weyl group. However, in our setting, since all maximal orbit types lie in $\Phi_o(O(2)\times S_4^p),$ so the computation can be simplified by restricting to the Burnside ring $A(O(2)\times S_4^p)$. More precisely, by applying the ring homomorphism \eqref{eq:pi_0-homomorphism} $\pi_0: U(G)\to A(G),$ we obtain the truncated bifurcation invariant as
\[
\tilde{\omega}_G(\lambda_{j,l})=\pi_0\Big(\omega_G(\lambda_{j,l})\Big).
\]
\vs
The following are the basic degrees truncated to $A(O(2)\times S_4^p)$ computed by G.A.P (see \cite{balanov2025degree}).
\begin{align}\label{eqn:basicdegree}
\nabla\text{-deg}_{\mathcal W_{0,l}}=&-\textcolor{red}{(D_l \times S^p_4)}+(O(2)\times  S_4^p)+\alpha_{0,l},\\
\nabla\text{-deg}_{\mathcal W_{4,l}}=&4({D_l}^{\mathbb Z_l}\times^{V_4^p}D_4^p)+({D_l}\times V_4^p)-\textcolor{red}{({D_{2l}}^{D_l}\times ^{V_4^p}D_4^p)}-\textcolor{red}{({D_l}\times D_4^p)}-2\textcolor{red}{({D_{3l}}^{\mathbb Z_l}\times^{V_4^p}_{D_3}S_4^p)}+(O(2)\times S_4^p)+\alpha_{4,l},\nonumber\\
%\mathrm{Deg}_{\mathcal{W}_{6,l}}=& -2({{D_l}}^{ \mathbb Z_l}\times _{\mathbb Z_1^p}D_2^d)-({D_l}\times \mathbb Z_1^p)+2({D_{2l}}^{\mathbb Z_l}\times_{\mathbb Z_1^p}D_4^z)_1+2({D_{2l}}^{\mathbb Z_l}\times _{\mathbb Z_1^p}D_4^z)_2+2({D_{2l}}^{\mathbb Z_l}\times _{\mathbb Z_1^p}V_4^p) \\
%& +2({D_{2l}}^{D_l}\times _{\mathbb Z_1^p}V_4)+({D_{2l}}^{D_l}\times _{\mathbb Z_1^p}D_2^d)-2\textcolor{red}{({D_{3l}}^{\mathbb Z_l}\times _{\mathbb Z_1^p}D_3^p)}-2\textcolor{red}{({D_{4l}}^{\mathbb Z_l}\times _{\mathbb Z_1^p}D_4^p)}-2\textcolor{red}{({D_{2l}}^{D_l}\times _{V_4}D_4^z)}\\
%&-\textcolor{red}{({D_{2l}}^{D_l}\times _{D_3^z}D_3^p)}-\textcolor{red}{({D_{2l}}^{D_l}\times _{D_4}D_4^p)}+(O(2)\times S_4^p),\\
\nabla\text{-deg}_{\mathcal W_{7,l}}=& -2({D_{2l}}^{\mathbb Z_l}\times _{\mathbb Z_2^p}\mathbb Z_2^p)-({D_{2l}}^{D_l}\times_{\mathbb Z_1^p} \mathbb Z_1^p)+2({D_{2l}}^{\mathbb Z_l}\times ^{D_1^z}_{D_2}D_2^p)+2({D_{2l}}^{\mathbb Z_l}\times ^{\mathbb Z_2^-}_{D_2}D_2^p)+2({D_{2l}}^{\mathbb Z_l}\times ^{\mathbb Z_2^-}_{D_2}V_4^p)\nonumber\\
&+2({D_{2l}}^{D_l}\times ^{D_1^z}D_1^p)+({D_{2l}}^{D_l}\times ^{\mathbb Z_2^-}\mathbb Z_2^p)-2\textcolor{red}{({D_{6l}}^{\mathbb Z_l}\times _{D_3^p}D_3^p)}-2\textcolor{red}{({D_{4l}}^{\mathbb Z_l}\times ^{\mathbb Z_2^-}D_4^p)}-\textcolor{red}{({D_{2l}}^{D_l}\times ^{D_2^d}D_2^p)}\nonumber\\
&-\textcolor{red}{({D_{2l}}^{D_l}\times ^{D_3^z}D_3^p)}-\textcolor{red}{({D_{2l}}^{D_l}\times ^{D_4^z}D_4^p)}+(O(2)\times S_4^p)+\alpha_{7,l},\nonumber
\end{align}
\begin{align}
\nabla\text{-deg}_{\mathcal W_{8,l}}=& -2({D_l}^{\mathbb Z_l}\times ^{\mathbb Z_1^p}\mathbb Z_2^p)-(D_l\times \mathbb Z_1^p)+2({D_l}^{\mathbb Z_l}\times ^{D_1^p}D_2^p)+2({D_l}^{\mathbb Z_l}\times ^{\mathbb Z_1^p}_{D_2}D_2^p)+2({D_{2l}}^{\mathbb Z_l}\times ^{\mathbb Z_1^p}_{V_4}V_4^p)\nonumber\\
&+2(D_l\times D_1^p)+({D_{2l}}^{D_l}\times ^{\mathbb Z_1^p}\mathbb Z_2^p)-2\textcolor{red}{({D_{3l}}^{\mathbb Z_l}\times_{D_3}D_3^p)}-2\textcolor{red}{({D_{4l}}^{\mathbb Z_l}\times_{D_4}D_4^p)}-\textcolor{red}{(D_{2l}^{D_l}\times ^{D_1^p}D_2^p)}\nonumber\\
&-\textcolor{red}{(D_l\times D_3^p)}
-\textcolor{red}{({D_{2l}}^{D_l}\times ^{D_2^p}D_4^p)}+(O(2)\times S_4^p)+\alpha_{8,l},\nonumber
\\
\nabla\text{-deg}_{\mathcal W_{9,l}}=& -2({D_{2l}}^{\mathbb Z_l}\times _{\mathbb Z_2^p}\mathbb Z_2^p)-(D_{2l}^{D_l}\times_{\mathbb Z_1^p} \mathbb Z_1^p)+2({D_{2l}}^{\mathbb Z_l}\times ^{D_1}_{\mathbb Z_2^p}D_2^p)+2({D_{2l}}^{\mathbb Z_l}\times ^{\mathbb Z_2^-}_{D_2}D_2^p)+2({D_{2l}}^{\mathbb Z_l}\times ^{\mathbb Z_2^-}_{D_2}V_4^p)\nonumber\\
&+2(D_{2l}^{D_l}\times ^{D_1}D_1^p)+({D_{2l}}^{D_l}\times ^{\mathbb Z_2^-}\mathbb Z_2^p)-2\textcolor{red}{({D_{6l}}^{\mathbb Z_l}\times _{D_3^p}D_3^p)}-2\textcolor{red}{({D_{4l}}^{\mathbb Z_l}\times^{\mathbb Z_2^-}D_4^p)}-\textcolor{red}{(D_{2l}^{D_l}\times ^{D_2^d}D_2^p)}\nonumber\\
&-\textcolor{red}{(D_{2l}^{D_l}\times ^{D_3} D_3^p)}
-\textcolor{red}{({D_{2l}}^{D_l}\times ^{D_4^d}D_4^p)}+(O(2)\times S_4^p)+\alpha_{9,l}\nonumber
\end{align}
where $\alpha_{j,l}$ does not belong to the Burnside Ring $A(O(2)\times S_4^p)$, i.e. $\text{coeff}^H(\alpha_{j,l})=0,\dim N(H)=0.$ Observe $\nabla\text{-deg}_{\mathcal W_{j,l}}$ is invertible in the Euler Ring $U(O(2)\times S_4^p).$\vs  

\subsection{Equivariant Bifurcation Result.}\label{sec:3.5}  Let's take Fourier mode $l=1$ for critical numbers $\lambda_{j,l}\in \Lambda,$  then we obtain the following result. 
\begin{theorem}\label{thm:main}
    Suppose critical numbers $\lambda_{j,1}\in \Lambda\; (j=0,4,7,7^*,8,9)$ are isotypic nonresonant, %i.e. $\lambda_{j,1}$ are isotypic simple. 
    the system \eqref{eqn01} admits branches of periodic solutions emerging from the equilibrium $v^o,$ with at least 16 different symmetry types of maximal orbit kinds  corresponding to irreducible representations $\mathcal W_j, j=0,4,7,7^*,8,9.$ More precisely,
    \begin{enumerate}
        \item[1)] j=0: $G$-orbit of branch of periodic solution with symmetry $(D_1\times S_4^p).$
        \item[2)] j=4: $G$-orbit of branches of periodic solutions with symmetries $(D_2^{D_1}\times ^{V_4^p}D_4^p), (D_1\times D_4^p), (D_3^{\mathbb Z_1}\times ^{V_4^p}_{D_3}S_4^p).$
        \item[3)] j=7 or $7^*$: $G$-orbit of branches of periodic solutions with symmetries $(D_6^{\mathbb Z_1}\times _{D_3^p}D_3^p), (D_4^{\mathbb Z_1}\times ^{\mathbb Z_2^-}D_4^p), (D_2 ^{D_1}\times^{D_2^d}D_2^p),(D_2 ^{D_1}\times^{D_3^z}D_3^p), (D_2 ^{D_1}\times^{D_4^z}D_4^p)$.
        \item[4) ] j=8: $G$-orbit of branches of periodic solutions with symmetries $(D_3^{\mathbb Z_1}\times _{D_3}D_3^p), (D_4^{\mathbb Z_1}\times _{D_4}D_4^p), (D_2 ^{D_1}\times^{D_1^p}D_2^p),(D_2 ^{D_1}\times^{D_2^p}D_4^p), (D_1\times D_3^p)$.
        \item[5) ] j=9: $G$-orbit of branches of periodic solutions with symmetries $(D_6^{\mathbb Z_1}\times _{D_3^p}D_3^p), (D_4^{\mathbb Z_1}\times ^{\mathbb Z_2^-}D_4^p), (D_2 ^{D_1}\times^{D_2^d}D_2^p),(D_2 ^{D_1}\times^{D_3}D_3^p), (D_2 ^{D_1}\times^{D_4^d}D_4^p)$.
    \end{enumerate}
\end{theorem}
\begin{proof}
    For $\lambda_{j_o,1}\in \Lambda,$ let $\lambda_-<\lambda_{j_o,1}<\lambda_+$ and $[\lambda_-, \lambda_+]\cap\Lambda=\{\lambda_{j_o,1}\}.$ From discussion in Section \ref{sec:3.2}, $\exists\; \mathscr U$ a tubular neighborhood of $G^(v^o)$ that is isolating so no other orbits from $J_{\lambda_{\pm}}$ is in $\Bar{\mathscr U}.$ Now we can apply Slice Principle (see Theorem \ref{thm:SCP}), i.e.
    \[
   \nabla_{G}\text{\textrm{-deg}}\Big(J_{\lambda_{\pm}}
,\mathscr U\Big)=\Theta\Big(\nabla_{G_{v^o}}\text{\textrm{-deg}}(J_{\lambda_{\pm}}
,\mathscr U\cap \bf S_o)\Big),
\]
where, in our case, $G=O(2)\times S_4^p\times O(3),\; G_{v^o}=O(2)\times S_4^p,$ and $\Theta: U(G_{v^o})\to U(G)$ is a homomorphism given by $\Theta(H)=(H),\; (H)\in \Phi_0(G).$ Therefore, the topological invariant now $\omega_G(\lambda_{j_o,1})$ takes the form
\begin{equation}\label{eqn:invariant}
\omega_G(\lambda_{j_o},1)=\nabla_{G_{v^o}}\text{\textrm{-deg}}(J_{\lambda_{-}}
,\mathscr U\cap \textbf {S}_o)-\nabla_{G_{v^o}}\text{\textrm{-deg}}(J_{\lambda_{+}}
,\mathscr U\cap \textbf {S}_o).
\end{equation}
We here omit $\Theta$ for notational convenience. \vs
On the other hand, we know that if $\omega_G(\lambda_{j_o},1)$ takes the form
\[
\omega_G(\lambda_{j_o},1)=n_1(H_1)+\cdots n_m(H_m),
\]
where $n_j\neq 0, \; (j=1,\cdots, m),$ then we say $\exists $ a branch of nontrivial solution bifurcating from $v^o$ with symmetry at least $(H_j).$ Our goal hereinafter is, for each $\lambda_{j_o,1}\in \Lambda$, to find the general form of $\omega_G(\lambda_{j_o},1)$ and more importantly, the coefficient $n_j$ for $(H_j).$ \vs
By linearization (see \eqref{eqn:linearization}), its computation based on $G$-equivariant basic degree \eqref{eq:grad-lin} and assumption of isotypic simplicity of critical numbers, one can derive the following
\[
\nabla_{G_{v^o}}\text{\textrm{-deg}}(J_{\lambda_{\pm}}
,\mathscr U\cap \textbf {S}_o)=\nabla_{G_{v^o}}\text{\textrm{-deg}}(\mathscr A_{\lambda_{\pm}}
,\mathscr U\cap \textbf {S}_o)
\]
and 
\[
\nabla_{G_{v^o}}\text{\textrm{-deg}}(\mathscr A_{\lambda_{-}}
,\mathscr U\cap \textbf {S}_o)=\prod_{\{(j,l)\in \mathbb N^2:\lambda_{j,l}<\lambda_{j_o,1}\}}\nabla\text{\textrm{-deg}}_{\mathcal W_{j,l}},
\]
\[
\nabla_{G_{v^o}}\text{\textrm{-deg}}(\mathscr A_{\lambda_{+}}
,\mathscr U\cap \textbf {S}_o)=\nabla\text{\textrm{-deg}}_{\mathcal W_{j_o,1}}\prod_{\{(j,l)\in \mathbb N^2:\lambda_{j,l}<\lambda_{j_o,1}\}}\nabla\text{\textrm{-deg}}_{\mathcal W_{j,l}}.
\]
Here, each $\nabla\text{\textrm{-deg}}_{\mathcal W_{j,l}}$ is the basic degree and has been directly computed by G.A.P. and shown in \eqref{eqn:basicdegree}.Thus by \eqref{eqn:invariant}, we obtain the following algorithm for computing $\omega_G(\lambda_{j_o},1).$
\begin{equation}\label{alg:algorithm_invariant}
\omega_G(\lambda_{j_o},1)=\prod_{\{(j,l)\in \mathbb N^2:\lambda_{j,l}<\lambda_{j_o,1}\}}\nabla\text{\textrm{-deg}}_{\mathcal W_{j,l}}\Big(O(2)\times S_4^p-\nabla\text{\textrm{-deg}}_{\mathcal W_{j_o,1}}\Big).
\end{equation}
More precisely,
\begin{align}\label{eq:invariant}
\omega_{G}(\lambda_{0,1}) &  =(O(2)\times S_{4}^p)-\nabla\text{\textrm{-deg}}_{\mathcal{W}_{0,1}%
},\\
\omega_{G}(\lambda_{7^*,1}) &  =\nabla\text{\textrm{-deg}}_{\mathcal{W}_{0,1}%
}\ast\Big((O(2)\times S_{4}^p)-\nabla\text{\textrm{-deg}}_{\mathcal{W}_{7^*,1}}\Big)\nonumber\\&=\nabla\text{\textrm{-deg}}_{\mathcal{W}_{0,1}%
}-\nabla\text{\textrm{-deg}}_{\mathcal{W}_{0,1}%
}\ast\nabla\text{\textrm{-deg}}_{\mathcal{W}_{7^*,1}},\nonumber\\
\omega_{G}(\lambda_{4,1}) &  =\nabla\text{\textrm{-deg}}_{\mathcal{W}_{0,1}%
}\ast\nabla\text{\textrm{-deg}}_{\mathcal{W}_{7*,1}}\ast\Big((O(2)\times S_{4}^p)-\nabla
\text{\textrm{-deg}}_{\mathcal{W}_{4,1}}\Big)\nonumber\\
&=\nabla\text{\textrm{-deg}}_{\mathcal{W}_{0,1}%
}\ast\nabla\text{\textrm{-deg}}_{\mathcal{W}_{7^*,1}}-\nabla\text{\textrm{-deg}}_{\mathcal{W}_{0,1}%
}\ast\nabla\text{\textrm{-deg}}_{\mathcal{W}_{7^*,1}}\ast\nabla
\text{\textrm{-deg}}_{\mathcal{W}_{4,1}},\nonumber\\
 \cdots& \nonumber
\end{align}
Indeed, we observe from \eqref{eqn:basicdegree} that all maximal orbit types are of zero dimensional weyl groups, so we can use \eqref{eqn:basicdegree} and algorithm \eqref{eq:invariant} to directly compute invariant $\omega_G(\lambda_{j_o},1),\; j_o=0,4,7,7^*,8,9.$ But notice that  $\lambda_{9,1}$ appears in the 29th in the sequence \eqref{sq:ordering}, the computation will introduce significant complexity. Given a maximal orbit type $(H)$ associated with a particular irreducible representation, in the following we examine the behavior of its coefficient in the invariant $\omega_G(\lambda_{j_o},1),$  in order to reduce this complexity.\vs
Suppose $(H)$ is a maximal orbit type in $\mathcal W_{j_o,1},$ i.e.
\[
\nabla
\text{\textrm{-deg}}_{\mathcal{W}_{j_o,1}}=O(2)\times S_4^p+n_H(H)+\cdots,
\]
Here $\cdots$ includes all submaximal. By Lemma \eqref{le:basic_coefficient},
\[
n_H=\begin{cases}
    -1,\;\;\; |W(H)|=2\\
    -2,\;\;\; |W(H)|=1.
\end{cases}
\]
Suppose $\nabla
\text{\textrm{-deg}}_{\mathcal{W}_{\tilde{j_o},1}}$ also contains $(H)$ with nonzero coefficient, then
\[
\nabla
\text{\textrm{-deg}}_{\mathcal{W}_{j_o,1}}*\nabla
\text{\textrm{-deg}}_{\mathcal{W}_{\tilde{j_o},1}}=O(2)\times S_4^p+2n_H(H)+n_H^2(H)*(H)+\cdots,
\]
where, from equation \eqref{eq:rec-coef}, $(H)*(H)=l_H(H)+\cdots,$ and $l_H$ takes the form
\[
l_H=\frac{n(H,H)|W(H)|n(H,H)|W(H)|}{|W(H)|}=|W(H)|.
\]
Therefore the coefficient of $(H)$ in this multiplication takes the form
\[
2n_H+n_H^2|W(H)|,
\]
which is equal to 0. Now, consider the bifurcation invarian $\omega_G(\lambda_{j_o,1})$ as defined in \eqref{alg:algorithm_invariant}. Suppose the orbit type  $(H)$ appears as a maximal type in the irreducible representation $\mathcal W_{j_o,1}$ with coefficient $n_H.$ Then $(H)$ also appears in the bifurcation invariant $\omega_G(\lambda_{j_o,1}),$  with coefficient $\tilde{\textbf{n}}_H$ given by 
\[
\tilde{\textbf{n}}_H=\begin{cases}
    \pm 1, \;\;\; |W(H)|=2\\
    \pm 2,\;\;\; |W(H)|=1.
\end{cases}
\]
\end{proof}\vs
One may refer to the following truncated invariant $\widetilde{\omega}_{G}(\lambda_{j_o,1})$ to $A(O(2)\times S_4^p)$ for further details. Note that all maximal isotropy types are highlighted in red.%
\begin{align*}
\widetilde{\omega}_{G}(\lambda_{0,1})= &  -{\color{red}(D_1\times S_4^p)},\\
\widetilde{\omega}_{G}(\lambda_{7^*,1})=& 2({D_1}^{\mathbb Z_1}\times _{\mathbb Z_1^p}\mathbb Z_1^p)+2({D_1}^{\mathbb Z_1}\times_{\mathbb Z_2^-}\mathbb Z_2^-)+2({D_1}^{\mathbb Z_1}\times _{D_1}D_1)+2({D_1}^{\mathbb Z_1}\times _{D_1^z}D_1^z)\\
&2({D_1}^{\mathbb Z_1}\times _{\mathbb Z_2}\mathbb Z_2)+({D_1}\times \mathbb Z_1)-2(D_1^{\mathbb Z_1}\times^{\mathbb Z_2^-} D_2^d)-4({D_1}^{\mathbb Z_1}\times ^{D_1^z}D_1^p)-2({D_1}^{\mathbb Z_1}\times ^{D_1^z}D_2^d)\\
&-2({D_2}^{\mathbb Z_1}\times_{\mathbb Z_2^p}\mathbb Z_2^p)-2({D_1}^{\mathbb Z_1}\times ^{\mathbb Z_2^-}\mathbb Z_2^p)-2({D_1}^{\mathbb Z_1}\times^{\mathbb Z_2^-}V_4^-)-2({D_1}^{\mathbb Z_1}\times^{D_1^z}D_2^z)+(D_1\times D_3^z)\\
&-2({D_2}^{D_1}\times _{\mathbb Z_1^P}\mathbb Z_1^p)+(D_1\times D_4^z)+2({D_1}^{\mathbb Z_1}\times^{D_3^z}D_3^p)\\
&+2({D_1}^{\mathbb Z_1}\times^{D_2^d}D_2^p)+2({D_2}^{\mathbb Z_1}\times^{\mathbb Z_2^-}_{D_2}V_4^p)+({D_2}^{D_1}\times^{\mathbb Z_2^-}\mathbb Z_2^p)
+2({D_2}^{\mathbb Z_1}\times ^{D_1^z}_{D_2}D_2^p)+2(D_2^{D_1}\times_{D_1^z}D_1^p)\\
&-(D_1\times \mathbb Z_2^-)-2(D_1\times D_1^z)-2\textcolor{red}{({D_6}^{\mathbb Z_1}\times_{D_3^p} D_3^p)}-2(\textcolor{red}{{D_4}^{\mathbb Z_1}\times ^{\mathbb Z_2^-}D_4^p})
-\textcolor{red}{(D_2^{D_1}\times ^{D_2^d}D_2^p)}\\
&-\textcolor{red}{(D_2^{D_1}\times ^{D_3^z}D_3^p)}-\textcolor{red}{(D_2^{D_1}\times ^{D_4^z}D_4^p)},\\
\widetilde{\omega}_{G}(\lambda_{4,1})=&-2({D_1}^{\mathbb Z_1}\times _{D_1}D_1)-2({D_1}^{\mathbb Z_1}\times_{D_1^z}D_1^z)-2({D_2}^{\mathbb Z_1}\times _{D_2^d}D_2^d)-2({D_2}^{\mathbb Z_1}\times _{D_1^p}D_1^p)+2(D_1^{\mathbb Z_1}\times ^{D_1^z}D_2^d)\\
&+2(\mathbb Z_1\times {D_2^d})-2({D_2}^{\mathbb Z_1}\times _{D_2}D_2)+2({D_1}^{\mathbb Z_1}\times ^{D_1^z}D_2^z)-({D_2}^{D_1}\times_{D_1}D_1)+(D_1\times D_1^z)\\
&+2({D_3}^{\mathbb Z_1}\times _{D_3}D_3)+2({D_3}^{\mathbb Z_1}\times _{D_3^z}D_3^z)+2({D_2}^{\mathbb Z_1}\times ^{\mathbb Z_2^-}_{D_2}D_2^p)-2({D_1}^{\mathbb Z_1}\times ^{D_2^d}D_2^p)+({D_2}^{D_1}\times ^{\mathbb Z_2^-}D_2^d)\\
&-(D_1\times D_2^d)+2({D_2}^{\mathbb Z_1}\times ^{V_4^-}_{D_2}D_4^p)-2({D_1}^{\mathbb Z_1}\times ^{D_4^z}D_4^p)+2({D_1}^{\mathbb Z_1}\times ^{V_4^p}D_4^p)+({D_2}^{D_1}\times ^{V_4^-}D_4^d)\\
&-(D_1\times D_4^z)-\textcolor{red}{({D_2}^{D_1}\times ^{V_4^p}D_4^p)}+\textcolor{red}{(D_1\times D_4^p)}-2\textcolor{red}{({D_3}^{\mathbb Z_1}\times ^{V_4^p}_{D_3}S_4^p)},\\
\cdots &
\end{align*}
\section{Description of Symmetries.}\label{sec:4}
%Recall amalgamated subgroups defined in\eqref{eq:def_almagameted}-\eqref{eq:amalgam_ker}. 
We now provide a detailed explanation of the 16 maximal orbit types. Below, the maximal orbit types $\mathscr H$ are categorized into three classes.
\begin{enumerate}
    \item[1.] $\mathscr H=D_1\times K,$ where $ K\leq S_4^p.$
    \item[2.] $\mathscr H=D_2^{D_1}\times^{K_o}_L K,$ where $K\leq S_4^p,$ and $ L=K/K_o$ isomorphic to $\mathbb Z_2$.
    \item[3.] $\mathscr H=H^{\mathbb Z_1}\times^{K_o}_L K,$ where $K\leq S_4^p,$ and $ L=K/K_o$ isomorphic to $H.$
\end{enumerate} 
Note that all generators presented in this section are derived from the subgroups computed by G.A.P (see \ref{app:subgroups}), together with the conversion given in \eqref{eq:s6s4p}. Moreover, for each maximal orbit type, we provide an animation. To generate this animation, one should refer to the concrete eigensystem listed in Table \eqref{tbl:eigensystem}.

\noi 1. Direct product $\mathscr H=D_1\times K,\; K\leq S_4^p.$
\begin{enumerate}
\item[(1)] $D_1\times S_4^p\; (j=0):$ This group is generated by $S^p_4$ and $\kappa\in D_1< O(2).$ The solution has full octahedral symmetries all the time.  In particular, they satisfy
\[
u(t)=\kappa u(t)=u(-t),\; \dot{u}(0)=\dot{u}(\pi)=0,
\]
meaning this is a brake orbit exhibiting a breathing mode. 
%Indeed, the solution has the general form
%\[
%u(x,t)=1.4128p^0+a\cos t \vec{v}+b\sin t\vec v,
%\]
%where $\vec v=(-1,0,0,0,-1,0,1,0,0,0,1,0,0,0,-1,0,0,1).$
The group’s explicit generators are:
\[
\Big(\kappa,1\Big), \;\Big(1, (1234)\Big),\; \Big(1, (146)(253)\Big),\Big(1, (13)(24)(56)\Big).
\]
\item[(2)]$D_1\times D_3^p\; (j=8):$ This group, like the previous one, defines a brake orbit and it preserves triangular symmetry throughout the motion. The generators are given by:
\[
\Big(\kappa, 1\Big), \Big(1,(346)(125)\Big),\Big(1,(14)(23)(56)\Big),\Big((1,(13)(24)(56)\Big).
\]
\item[(3)]$D_1\times D_4^p\;(j=4):$ Another brake orbit. Position of atom 1 decides atom 2, 3, 4 simultaneously, atom 5 decides atom 6, but there is no relation between these two pairs. Generators are
\[
\Big(\kappa, 1\Big), \Big(1,(1234)\Big),\Big(1,(14)(23)(56)\Big),\Big((1,(13)(24)(56)\Big).
\]
\end{enumerate}
%by which we have that $b=0.$
\vs
2. Amalgamated group $\mathscr H=D_2^{D_1}\times^{K_o}_L K,$ where $K\leq S_4^p.$
\begin{enumerate}
    \item $D_2^{D_1}\times ^{D_2^d}D_2^p\;(j=9):$ Notice that $D_2^d=<(56),(14)(23)(56)>$ and $D_2^p/D_2^d\cong <(12)(34)(56)>,$ so the group is generated by 
    \[
    \kappa\in D_1< O(2); \; (14)(23)(56),(56)\in D_2^d < S_4^p; \; \Big(e^{i\pi},(12)(34)(56)\Big)\in O(2)\times S_4^p.
    \]
 It's still a brake orbit. Here $M_{(14)(23)(56)}$ 
  denotes the reflection through the plane containing points $p_1,p_2,p_3,$ with respect to the axis connecting the midpoints of segments $p_1 p_4$, $p_2p_3.$ So $v_1$ is an inversion of $v_4,$ $v_2$ is an inversion of $v_3$ and $v_5$ is an inversion of $v_6$. Moreover, there are 3 pairs of $\pi$ rotation for $u_1$ and $u_2$, $u_3$ and $u_4$, $u_5$ and $u_6$, together with $\pi$ time shift.\vs
  \item $D_2^{D_1}\times ^{D_1^p}D_2^p\;(j=8):$ $D_1^p=<(14)(23)(56),(13)(24)(56)>$ and $D_2^p/D_1^p\cong\mathbb Z_2=<(13)(24)>.$ Thus, the generators for this brake orbit are:
  \[
  \Big(\kappa,1\Big),\Big(1,(14)(23)(56)\Big),\Big(1,(13)(24)(56)\Big),\Big(e^{i\pi},(13)(24)\Big).
  \]
 % and $v_1$ simultaneously decides atom $v_2,v_3,v_4$. Moreover, $v_5$ is an inversion of $v_6.$  
  Here, $M_{(13)(24)}$ represents a $\pi$ -rotation about the axis through points $p_5p_6,$  yielding two coupled pairs of 
 $\pi$ rotation: one pair acting on $v_1$, $v_3$, another pair acting on $v_2$, $v_4$, each coupled to a $\pi$-phase shift. \vs
    \item $D_2^{D_1}\times ^{D_3^z}D_3^p\;(j=7): $ Notice that $D_3^z$ is generated by $<(346)(125),(12)(34)>$, $D_3^p$ generated by $<(346)(125),(14)(23)(56),(13)(24)(56)>$ and $D_3^p/D_3^z\cong<(13)(24)(56)>$. Thus one can easily get the generators of the whole group: 
    \[
    \Big(\kappa,1\Big),\Big(1,(346)(125)\Big),\Big(1,(12)(34)\Big),\Big(e^{i\pi},(13)(24)(56)\Big).
    \]
 Obviously, another brake orbit. Here $M_{(125)(346)}$ denotes the rotation by $2\pi/3$ around the axis passing through the centers of the triangular faces formed by vertices $p_1p_2p_5$ and $p_3p_4p_6.$ The antipodal reflection of the configuration is coupled with a $\pi$-phase shift.\vs
 \item $D_2^{D_1}\times ^{D_3}D_3^p\;(j=9): $ Similarly, the solution has triangular symmetry for $D_3,$ and the antipodal reflection is coupled with a $\pi$-phase shift. See generators of the group:
 \[
    \Big(\kappa,1\Big),\Big(1,(346)(125)\Big),\Big(1,(14)(23)(56)\Big),\Big(e^{i\pi},(13)(24)(56)\Big).
 \]
    \item $D_2^{D_1}\times ^{D_4^z}D_4^p:$ $D_4^p$ is generated by $(14)(23)(56),(13)(24)(56),(1234),$ and $D_4^z=<(12)(34),(1234)>$. Since $D_4^p/D_4^z=<(13)(56)>,$ one can easily derive the following conclusion: $D_2^{D_1}\times ^{D_4^z}D_4^p$ admits the generators 
    \[
    \Big(\kappa,1\Big), \; \Big(1,(12)(34)\Big),\;\Big(1,(1234)\Big),\; \Big(e^{i\pi},(13)(56)\Big)
    \]
%    which means 
%    \[
%    \begin{cases}
%        u_1(t)=u_2(-t+\pi)=u_3(t)=u_4(-t+\pi)\\
%        u_5(t)=u_6(-t)
%    \end{cases}
%    \]
Here $M_{(1234)}$ denotes the $\pi/2$ rotation of the plane $p_1p_2p_3p_4$ and these four atoms interchange. $M_{(13)(56)}$ is the rotation by $\pi$ along the axis of $p_2p_4,$ so that $v_1$ is an inversion of $v_3$, $v_5$ is an inversion of $v_6$, both of which are coupled with a $\pi$-phase shift. Notice that there exists two distinct pairs, i.e. $v_1$ decides $v_2,v_3,v_4$, and $v_6$ is only decided by $v_5$.\vs
 \item ${D_2}^{D_1}\times ^{V_4^p}D_4^p\;(j=4):$ The solutions have symmetry of $V_4.$ Notice that $D_4^p/V_4^p=<(14)(23)(56)>,$  the spatial rotation by $\pi$ around the axis connecting the midpoints of segments  $p_1p_4$ and $p_2p_3$ is combined with a temporal shift of $\pi$.\vs
 %Notice $V_4^p=<(13)(56),(24)(56),(13)(24)(56)>$ and $D_4^p/V_4^p\cong D_1=<(14)(23)(56)>.$ So the whole group has following generators: $\kappa \in O(2);\; (13)(24)(56),(24)(56),(13)(56)\in V_4^p$ and $\Big(e^{i\pi},(14)(23)(56)\Big)$. i.e. $u_1=u_3$ are the inversion of $u_2=u_4$ while there also exists $\pi$ time shift between $u_2$ and $u_4.$  Moreover, $u_5$ is also inversion of $u_6.$ So here $u_1$ decides all $u_2, u_3, u_4$, but this pair has no relation with $u_5$ and $u_6.$

 \item $D_2^{D_1}\times ^{D_4^d}D_4^p\;(j=9):$ The generators are
 \[
 \Big(\kappa,1\Big),\Big(1,(14)(23)(56)\Big),\Big(1,(24)\Big),\Big(1,(13)(24)\Big),\Big(e^{i\pi},(13)(56)\Big).
 \]
 The $\pi$ rotation along axis $p_1p_3$ is coupled with $\pi$ phase shift. Once position of $v_1$ is known, then $v_2,v_3,v_4$ can be determined, $v_5$ and $v_6$ is another pair. It is, again, brake orbit.\vs
  \item $D_2^{D_1}\times ^{D_2^p}D_4^p\;(j=8):$ The generators are
 \[
 \Big(\kappa,1\Big),\Big(1,(14)(23)(56)\Big),\Big(1,(13)(24)\Big),\Big(e^{i\pi},(13)(56)\Big).
 \]
 Similarly, the solutions are symmetric by $D_2,$ while the reflection of $D_4$ is coupled with $\pi$ phase shift, and $u_1(t)=M_{(13)(56)}u_3(t+\pi),\;u_5(t)=M_{(13)(56)}u_6(t+\pi).$
\end{enumerate}\vs

\noi 3. Amalgamated group $\mathscr H=H^{\mathbb Z_1}\times^{K_o}_L K,$ where $K\leq S_4^p.$
\begin{enumerate}
\item $D_6^{\mathbb Z_1}\times _{D_3^p}D_3^p\;(j=9):$ Consider the spatial group \[
D_3^p=<(346)(125),(14)(23)(56),(13)(24)(56)>, 
\]
and the temporal $D_6$. We define $\mathscr H$ such that each temporal symmetry corresponds exactly to a spatial one. Generators of $\mathscr H$ are  
\[\Big<\Big(e^{i\frac{\pi}{3}},(145326)\Big), \Big(\kappa,(14)(23)(56)\Big)\Big>.
\]
\begin{comment}
Indeed, one has the following correspondence relationship:
\begin{center}
\label{eq:table}
\resizebox{\textwidth}{!}{
\begin{tabular}{|c||c|c|c|c|c|c|c|c|c|c|c|c|}\hline
$D_6$& $1$& $e^{i\frac{\pi}{3}}$& $e^{i\frac{2\pi}{3}}$& $e^{i\pi}$& $e^{i\frac{4\pi}{3}}$& $e^{i\frac{5\pi}{3}}$&$\kappa$& $\kappa e^{i\frac{\pi}{3}}$& $\kappa e^{i\frac{2\pi}{3}}$ & $\kappa e^{i\pi}$ & $\kappa e^{i\frac{4\pi}{3}}$& $\kappa e^{i\frac{5\pi}{3}}$\\\hline$D_3^p$ & $(1,1)$& $(r^2,-1)$& $(r,1)$&$(1,-1)$ & $(r^2,1)$ & $(r,-1)$ & $(\tilde\kappa,1)$& $(\tilde \kappa r^2,-1)$& $(\tilde \kappa r,1)$& $(\tilde \kappa,-1)$& $(\tilde \kappa r^2,1)$ & $(\tilde \kappa r,-1)$\\\hline
&$(1,1)$&$(234)(56)$&$(243)$&$(56)$&$(234)$&$(243)(56)$&$(24)$&$(23)(56)$&$(34)$&$(24)(56)$&$(23)$&$(34)(56)$\\\hline
\end{tabular}}
\end{center} 
%Therefore we %have
%\[
%\begin{cases}
%u_2(t)=u_3(t+\frac{\pi}{3})=u_4(t+\frac{2\pi}{3}) \\
%u_5(t)=u_6(t+\frac{\pi}{3})\\
%u_2(t)=u_4(-t)=u_3(-t-%\frac{2\pi}{3})
%\end{cases},
%\]
which represents a rotating wave.
\end{comment}
Therefore, $v_1(t)\; (resp.\; v_2(t),\; v_5(t))$ is an inversion of $v_4(-t)\; (resp.\; v_3(-t), \; v_6(-t)).$ There is also a $\pi/3$-rotoreflection coupled with a temporal phase shift by $\pi/3.$\vs
\item $D_4^{\mathbb Z_1}\times ^{\mathbb Z_2^-}D_4^p\;(j=9):$ $D_4^p$ is generated by $<(1234),(14)(23)(56),(13)(24)(56)>,$ where $(1234)$ denotes rotation of particle $p_1,\cdots, p_4$, $(14)(23)(56)$ the reflection along midpoints of $p_1p_4,p_2p_3$ axis and $(13)(24)(56)$ the antipodal reflection. Notice that the configuration of this molecule has symmetry $\mathbb Z_2^-=<(56)>$ all the time,  and the generators are given by
\begin{equation}\label{eq:generator}
\Big<\Big(e^{i\frac{\pi}{2}},(1234)\Big),\Big(\kappa, (14)(23)(56)\Big),\Big(e,(56)\Big)\Big>.
\end{equation}
So $v_5$ and $v_6$ always share the same dynamics and $\Big(\kappa,(14)(23)(56)\Big)$ shows the same dynamics as above. There is also a $\pi/2$ -rotoreflection among atoms $v_1,v_2,v_3,v_4$ with $\pi/2$ phase shift.
\vs
\item $D_3^{\mathbb Z_1}\times _{D_3}^{V_4^p}S_4^p\;(j=4):$ Suposse $V_4^p$ generated by $\Big<(24)(56),(13)(56),(13)(24)(56)\Big>$and spatial $D_3$ generated by $\Big<(346)(125),(14)(23)(56)\Big>$ while temporal $D_3$ with the corresponding generators $\Big<e^{i\frac{2\pi}{3}},\kappa\Big>$. So the generators are given by
\[
\Big<\Big(e^{i\frac{2\pi}{3}},(346)(125)\Big), \; \Big(\kappa, (14)(23)(56)\Big),\Big(e,(24)(56)\Big),\Big(e,(13)(56)\Big),\Big(e,(13)(24)(56)\Big)\Big>
\]
one can conclude that in these solutions, we have three pairs of particles with identical dynamics all the time, i.e. $v_1 \& v_3$, $v_2 \& v_4$, $v_5\& v_6$. Moreover, $v_1(t)$ is an inversion of $v_4(-t)$, $v_2(t)$ is an inversion of $v_3(-t).$ The $2\pi/3$ rotation around
the axis passing through the centers of the triangular faces formed by vertices $p_1p_2p_5$ and
$p_3p_4p_6$ is coupled with temporal time shift by $2\pi/3.$\vs
\item $D_3^{\mathbb Z_1}\times _{D_3}D_3^p\;(j=8):$   The generators are given by
\[
\Big<\Big(e^{i\frac{2\pi}{3}},(346)(125)\Big), \; \Big(\kappa, (14)(23)(56)\Big),\Big(e,(13)(24)(56)\Big)\Big>
\]
Similar to above, the antipodal particles always share the same dynamics, the $2\pi/3$-phase shift is coupled with $2\pi/3$ rotoreflection, and spatial reflection across the axis connecting the midpoints of segments $p_1p_4$ and $p_2p_3$.\vs
\item $D_4^{\mathbb Z_1}\times _{D_4}D_4^p\;(j=8):$   The generators are given by
\[
\Big<\Big(e^{i\frac{\pi}{2}},(1234)\Big), \; \Big(\kappa, (14)(23)(56)\Big),\Big(e,(13)(24)(56)\Big)\Big>
\]
So the $\pi/2$ rotoreflection of $v_1,v_2,v_3,v_4$ is coupled with $\pi/2$ time shift. The solution is governed by
\[
\begin{cases}
u_1(t)=M_{(14)(23)(56)}u_4(-t),\;u_2(t)=M_{(14)(23)(56)}u_3(-t),\; u_5(t)=M_{(14)(23)(56)}u_6(-t),\\
u_1(t)=M_{(13)(24)(56)}u_3(t),\;u_2(t)=M_{(13)(24)(56)}u_4(t),\; u_5(t)=M_{(13)(24)(56)}u_6(t),\\
u_1(t)=M_{(1234)}u_4(t+\pi/2),\;u_2(t)=M_{(1234)}u_1(t+\pi/2),\; u_3(t)=M_{(1234)}u_2(t+\pi/2).
\end{cases}
\]
\end{enumerate}

We compute the explicit linearized solutions for all detected symmetries. See Fig \eqref{gif:Dynamics2} for part of the symmetries detected. We show in the right column the norms of particles oscillating in time. Note that modes (A) and (B)  correspond, respectively, to the dynamics presented in Figures 2(a) and (b) of \cite{chechin2015nonlinear}, which characterize breathing modes. %Moreover, $\mathcal W_7$ corresponds to translation of all six F particles, as described in \cite{chechin2015nonlinear}. It can also be verified by the animation for $j=7.$

\begin{figure}[H]
    %\centering
\begin{subfigure}{\textwidth}
\centering
\includegraphics[width=0.25\textwidth]{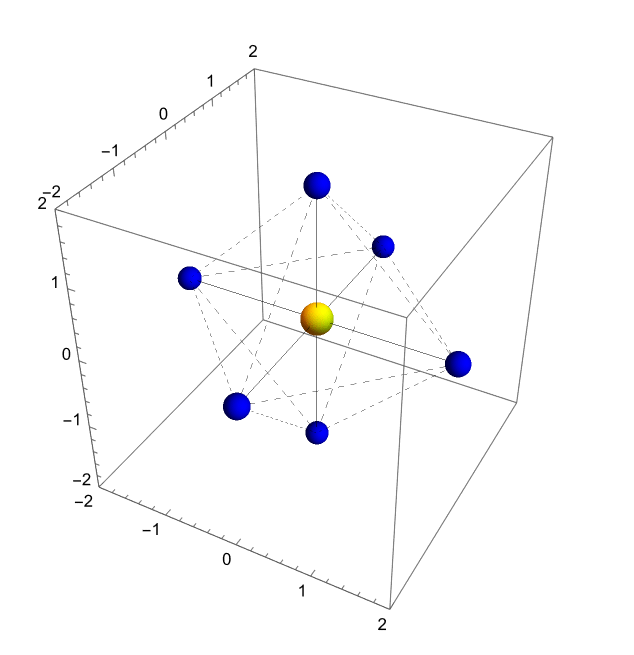}
  \hspace{0.1\textwidth}
\includegraphics[width=0.4\textwidth]{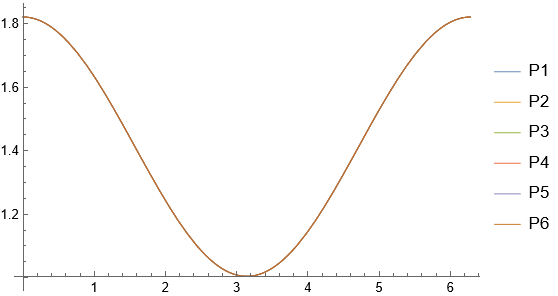}\label{gif:breathing}
\caption{ $D_1 \times S_4^p\;(j=0)$}
\end{subfigure}

\par\vspace{0.3cm}\par
\begin{subfigure}{\textwidth}
\centering
\includegraphics[width=0.25\textwidth]{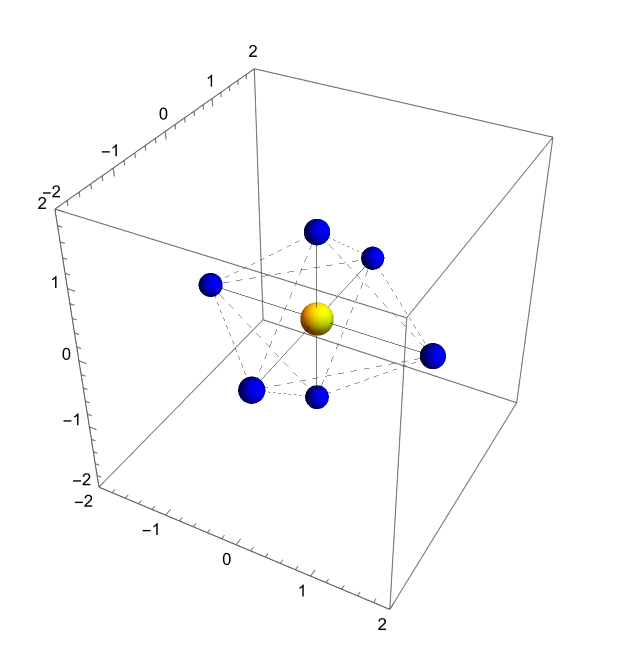}
\hspace{0.1\textwidth}
\includegraphics[width=0.4\textwidth]{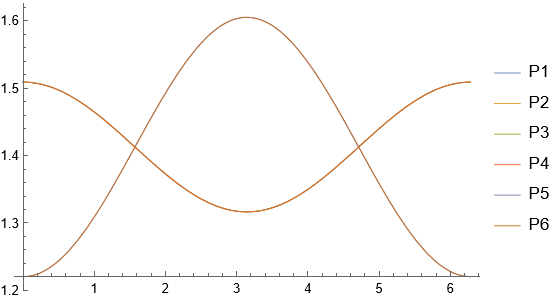}
\caption{\small $D_1 \times D_4^p\;(j=4)$}
\end{subfigure}
\par\vspace{0.3cm}\par

\begin{subfigure}{\textwidth}
\centering
\includegraphics[width=0.25\textwidth]{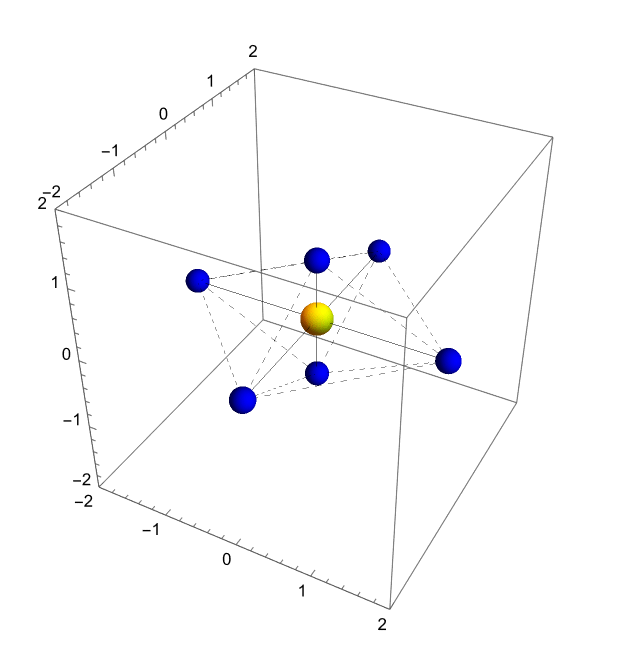}
\hspace{0.1\textwidth}
\includegraphics[width=0.4\textwidth]
{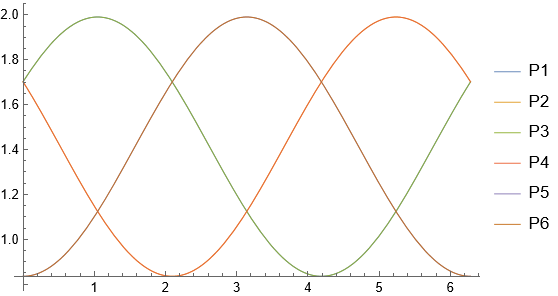}
\caption*{\small (C) $D_3^{\mathbb Z_1}\times^{V_4^p}S_4^p\;(j=4)$}
\end{subfigure}
\par\vspace{0.4cm}\par
\begin{subfigure}{\textwidth}
    \centering    \includegraphics[width=0.25\textwidth]{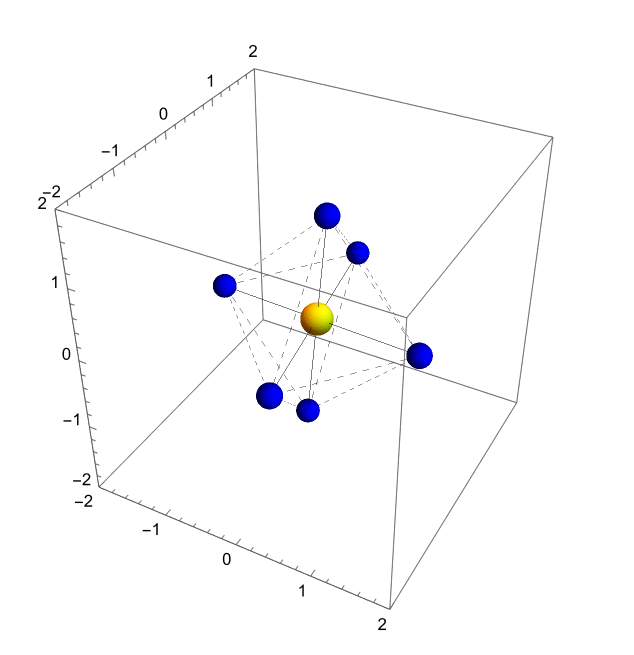}
     \hspace{0.1\textwidth}   \includegraphics[width=0.4\textwidth]{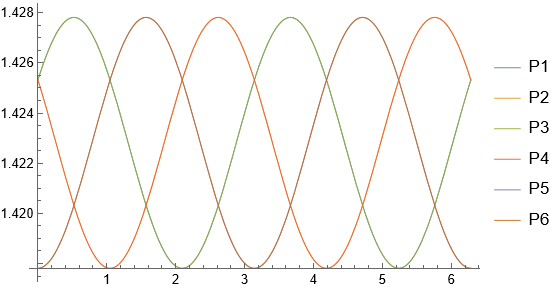}
        \caption*{\small (D) $D_3^{\mathbb Z_1}\times_{D_3}D_3^p\;(j=8)$}
\end{subfigure}    
\end{figure}
\par\vspace{0.3cm}\par

\begin{figure}[H]
\begin{subfigure}
{\textwidth}
\centering
\includegraphics[width=0.25\textwidth]{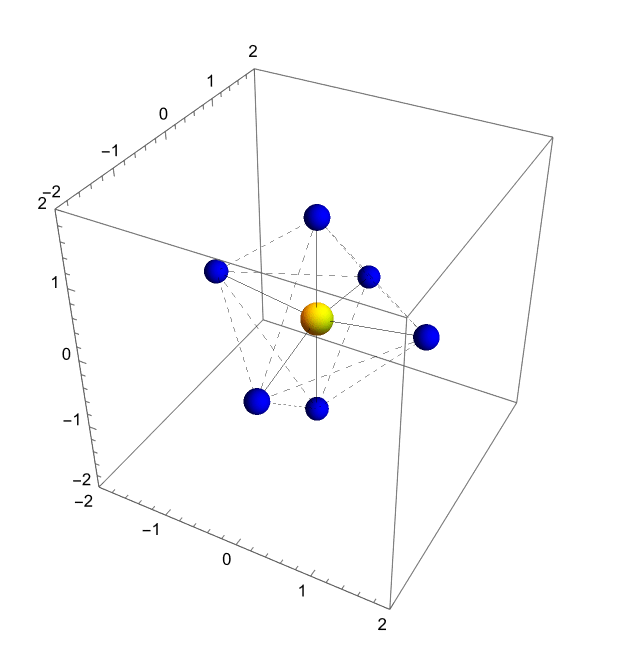}
\hspace{0.1\textwidth}       \includegraphics[width=0.4\textwidth]{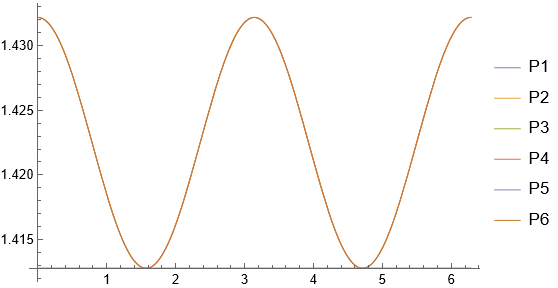}
\caption*{\small (E) $D_2^{D_1}\times ^{D_4^d}D_4^p\;(j=9)$}
\end{subfigure}
\par\vspace{0.3cm}\par

\begin{subfigure}{\textwidth}
\centering
\includegraphics[width=0.22\textwidth]{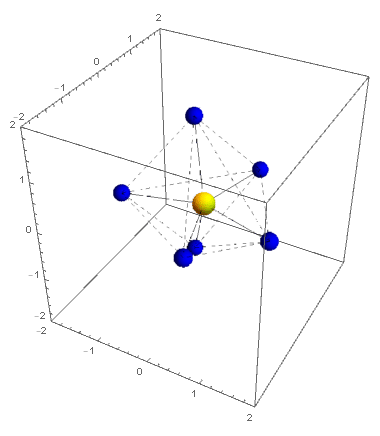}
\hspace{0.1\textwidth} 
\includegraphics[width=0.4\textwidth]{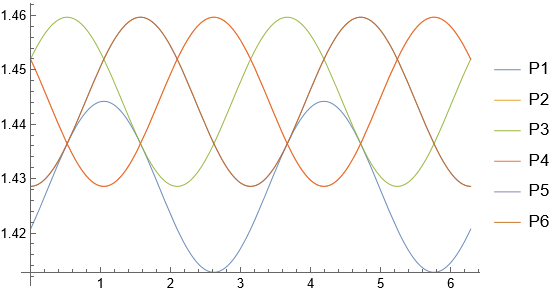}
\caption*{\small (F)$\;D_6^{\mathbb Z_1}\times_{D_3^p}D_3^p\;(j=9)$}
\end{subfigure}

\par\vspace{0.3cm}\par

\begin{subfigure}{\textwidth}
\centering
\includegraphics[width=0.22\textwidth]{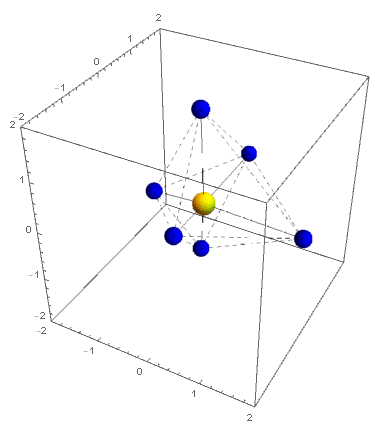}
\hspace{0.1\textwidth}         \includegraphics[width=0.4\textwidth]{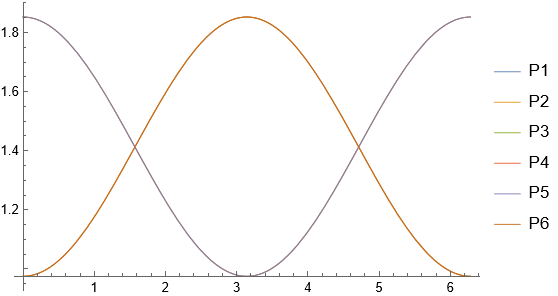}
\caption*{\small (G)$\;D_2^{D_1}\times^{D_3^z}D_3^p\;(j=7^*)$}
\end{subfigure}

\par\vspace{0.2cm}\par

\begin{subfigure}{\textwidth}
        \centering        \includegraphics[width=0.22\textwidth]{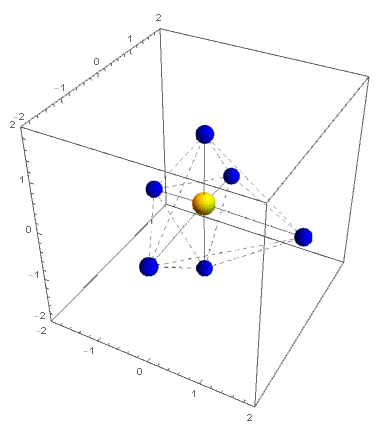}
  \hspace{0.1\textwidth}     \includegraphics[width=0.4\textwidth]{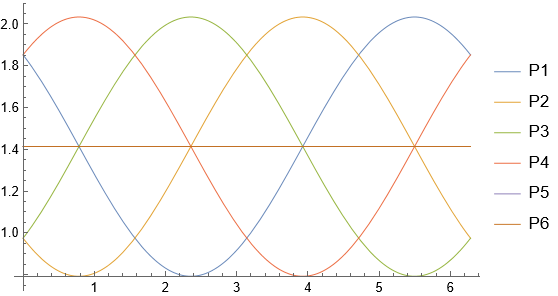}
        \caption*{\small (H)$\;D_4^{\mathbb Z_1}\times ^{\mathbb Z_2^-}D_4^p\;(j=7^*)$}
    \end{subfigure}
\par\vspace{0.3cm}\par
 \begin{subfigure}{\textwidth}
        \centering
        \includegraphics[width=0.25\textwidth]{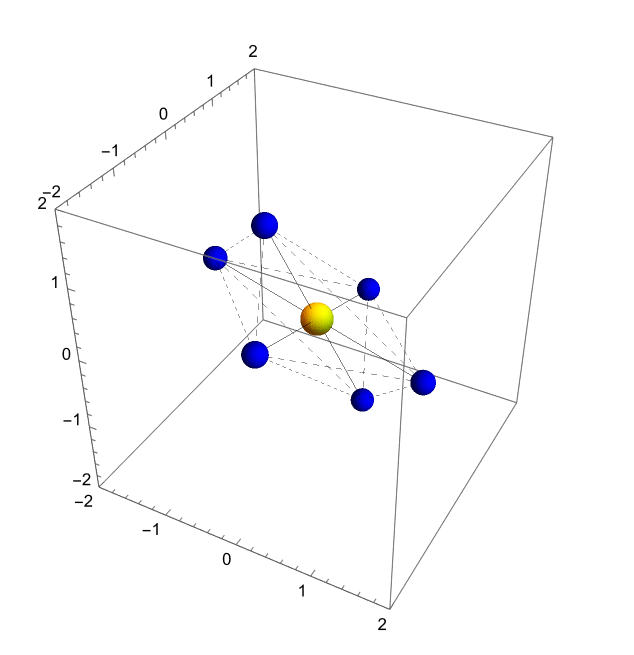}
  \hspace{0.1\textwidth}           \includegraphics[width=0.4\textwidth]{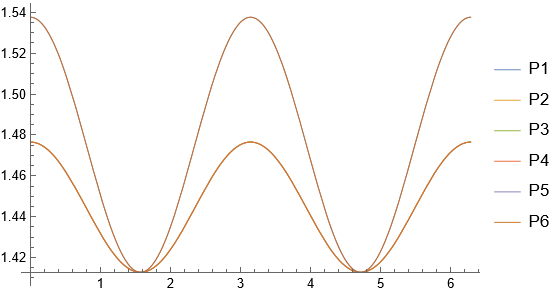}
        \caption*{\small (I) $\;D_2^{D_1}\times^{D_1^p}D_2^p\;(j=8)$}
    \end{subfigure}

    \caption{\small Vibrational modes }
    \label{gif:Dynamics2}
\end{figure}\vs
\noi  For animated illustrations of all vibrational modes, please refer to: 
\url{https://github.com/violal1016/Octahedral-Molecule/blob/main/Dynamics_Animation/j%3D8/D1xD3p_(j%3D8).gif} 

\begin{comment}
solution has symmetry with $u_1(t)=u_2(t)=u_3(t)=u_4(t),\; u_5(t)=u_6(t)$ at any time and
\[
\begin{cases}
u_2(t)=u_4(t+\frac{2\pi}{3})=u_3(t+\frac{4\pi}{3})=u_4(t+\frac{4\pi}{3})\\
u_4(t)=u_3(t)=u_4(-t)
\end{cases},
\]
which is another rotating wave.
\end{comment}

\appendix
\section{Conjugacy classes of $S_4^p$}\label{app:A}

\subsection{Conjugacy classes of $S_4^p$}
\label{app:subgroups}
Below, we introduce notation for representatives of the conjugacy classes of subgroups in $S_{4}$:
\begin{align}\label{eq:con_S4}
{\mathbb{Z}}_{1}  &  =<(1)>,\quad{\mathbb{Z}}_{2}=<(12)(34)>,\quad
{\mathbb{Z}}_{3}   =<(234)>,\quad V_{4}
=<(14)(23),(12)(34)>,\\
A_{4}  &
=<(14)(23),(12)(34),(234)>,\quad
D_{4}=<(1234),(24)>,\quad
{\mathbb{Z}}_{4}   =<(1234)>,\nonumber\\
D_{3}  &  =<(234),(24)>,\quad
D_{2}  =<(13)(24),(24)>,\quad
D_{1}   =<(24)>,\nonumber\\
S_4&=<(14)(23),(13)(24),(234),(24)>.\nonumber
\end{align}
Then, up to conjugacy, the subgroups of $S_4^p=S_4\times{\mathbb{Z}}_{2}$,
fall into two primary categories: \newline(i) Product subgroups of the form
$H\times{\mathbb{Z}}_{2}$, where $H$ is a subgroup of $S_4$;
\newline(ii) Twisted subgroups defined by 
\[
H^{\varphi}:=\{(g,z)\in S_4\times{\mathbb{Z}}_{2}:\varphi(g)=z\},
\]
where $\varphi:H\rightarrow{\mathbb{Z}}_{2}$ is a
homomorphism and $H\leq S_4$. \vs
To be more precise, when $\varphi:H\rightarrow{\mathbb{Z}}_{2}$ is a trivial homomorphism, the resulting twisted subgroup $H^{\varphi}$ coincides with the subgroups $H$ of $S_4$. This covers subgroups of the type $\mathbb Z_n (n=1,2,3,4)$, $V_4$, $A_4$, $D_{n} (n=1,2,3,4)$, $S_4.$ We also introduce the following twisted subgroups:
\begin{enumerate}[label={\textbullet}]
\item $D_n^z,$ where
$z:D_{n}\rightarrow{\mathbb{Z}}_{2}$ denotes the epimorphism satisfying
$\ker(z)={\mathbb{Z}}_{n}$, 
\item $D_{2n}^d,$ where $d:D_{2n}\rightarrow{\mathbb{Z}}_{2}$ denotes the
epimorphism with $\ker(d)=D_{n}$,
\item $\mathbb Z_{2n}^d,$ where $d:{\mathbb{Z}}_{2n}\rightarrow{\mathbb{Z}
}_{2}$ denotes the epimorphism with $\ker(d)={\mathbb{Z}}_{n}$, 
\item $S_4^-,$ where ${\ -}$ denotes the
epimorphism $\varphi:S_{4}\rightarrow{\mathbb{Z}}_{2}$ with $\ker
(\varphi)=A_{4}$,
\item $V_4^-$, where $-$ denotes the epimorphisms $\varphi
:V_{4}\rightarrow{\mathbb{Z}}_{2}$ with $\ker(\varphi)={\mathbb{Z}}_{2}$.
\end{enumerate}
%Tables \ref{tab:product-s4-z2} and
%\ref{tab:twisted-s4-z2} lists the product subgroups $H^p$
%and twisted subgroups $H^{\varphi}$ in $S_4^p,$
%and their Weyl groups.  
In addition to subgroups in \eqref{eq:con_S4}, we also list the other 22 representative of conjugacy classes of subgroups in $S_4^p$ as follows, which are computed by G.A.P:
\begin{align*}
\mathbb Z_1^p&=<(56)>,\quad D_1^p=<(24),(56)>,\quad D_1^z=<(24)(56)>,\nonumber\\   
\mathbb Z_2^p &=<(12)(34),(56)>,\quad \mathbb Z_2^-=<(12)(34)(56)>,\nonumber\\
D_2^p&=<(13)(24),(24),(56)>,\quad D_2^z=<(13)(24),(24)(56)>,\nonumber\\
D_2^d&=<(13)(24)(56),(24)>,\quad \mathbb Z_4^d=<(1234)(56),(13)(24)>,\nonumber\\
V_4^-&=<(14)(23)(56),(12)(34)>,\quad
V_4^p=<(14)(23),(12)(34),(56)>,\\
D_4^z&=<(1234),(24)(56)>,\quad
D_4^d=<(13)(24),(24),(12)(34)(56)>,\\
D_4^{\tilde d}&=<(14)(23),(12)(34),(24)(56)>,\quad D_4^p=<(1234),(24),(56)>\nonumber\\
\mathbb Z_3^p&=<(234),(56)>,\quad
D_3^z=<(234),(24)(56)>,\quad D_3^p=<(234),(24),(56)>,\nonumber\\
A_4^p&=<(14)(23),(12)(34),(234),(56)>,\quad
S_4^-=<(14)(23),(12)(34),(234),(24)(56)>,\nonumber\\
S_4^p&=<(14)(23),(12)(34),(234),(24),(56)>,\quad \mathbb Z_4^p=<(1234),(56)>.\nonumber
\end{align*}
Note, we use the amalgamated notation in this work to describe subgroups of $G=O(2)\times S_4^p.$ For details of subgroups of $O(2)$ and explanation of amalgamated subgroups in direct product group $G_1\times G_2,$ one can refer to Chapter 5.3 of \cite{balanov2006applied} and Appendix E of \cite{balanov2025degree}, repsectively.
\section{Euler and Burnside Rings}\label{app:C}
\begin{comment}
\noindent\textbf{Euler Characteristic:} For a topological space $Y$, denote by
$H_{c}^{\ast}(Y)$ the ring of Alexander-Spanier cohomology with compact
support (see \cite{Spa}). If $H_{c}^{\ast}(Y)$ is finitely generated, then the
Euler characteristic $\chi_{c}(Y)$ is well defined. For a compact CW-complex
$X$ and its closed subcomplex $A$, as it is very well-known $H_{c}^{\ast
}(X\setminus A)\cong H^{\ast}(X,A;\mathbb{R})$, where $H^{\ast}(\cdot)$ stands
for a usual cellular cohomology ring (see \cite{Spa}, Ch.~6, Sect.~6, Lemma
11). Therefore, $\chi_{c}(X\setminus A)$ is correctly defined and
\[
\chi(X)=\chi_{c}(X\setminus A)+\chi(A)=\chi(X,A)+\chi(A).
\]
where $\chi(\cdot)$ stands for the Euler characteristic with respect to the
cellular cohomology groups. Moreover, if $Y$ is a compact $CW$-complex,
$B\subseteq Y$ is a closed subcomplex and $p:X\setminus A\rightarrow
Y\setminus B$ a locally trivial fibre bundle with path-connected base and
fibre $F$ is a compact manifold, then (cf. \cite{Spa}, Ch. 9, Sect. 3, Theorem
1)
\[
\chi_{c}(X\setminus A)=\chi(F)\chi_{c}(Y\setminus B).
\]
\end{comment}
\vskip.3cm

\noindent\textbf{Euler and Burnside Rings\label{app:Euler}}

In this section we assume that $G$ stands for a compact Lie group and we
denote by $\Phi(G)$ the set of all conjugacy classes $(H)$ of closed subgroups
$H$ of $G$. For any $(H)\in\Phi(G)$ we denote by $N(H)$ the normalizer of $H$
and by $W(H):=N(H)/H$ the Weyl's group of $H$. \vskip.3cm

Notice that $\Phi(G)$ admits a natural order relation given by
\begin{equation}
\label{eq:order}(K)\le(H) \;\; \Leftrightarrow\;\; \exists_{g\in G}\; \;
gKg^{-1}\subset H, \;\; \text{ for }\;\; (K),\,(H)\in\Phi(G).
\end{equation}
Moreover, we define for $n=0,1,2,\dots$ the following subsets $\Phi_{n}(G)$ of
$\Phi(G)$
\[
\Phi_{n}(G):=\{ (H)\in\Phi(G): \dim W(H)=n\}.
\]
\vskip.3cm \noi Let  
$U(G)={\mathbb{Z}}[\Phi(G)]$ be the free ${\mathbb{Z}}$-module generated by $\Phi(G)$, then an element $a\in U(G)$ is represented as
\begin{equation}
\label{eq:sum-a}a=\sum_{(L)\in\Phi(G)} n_{L} \, (L), \quad n_{L}\in
{\mathbb{Z}},
\end{equation}
where the integers $n_{L}=0$ except for a finite number of elements
$(L)\in\Phi(G)$. For such element $a\in U(G)$ and $(H)\in\Phi(G)$, we will
also use the notation
\begin{equation}
\label{eq:coeff-H}\text{coeff}^{H}(a)=n_{H},
\end{equation}
i.e. $n_{H}$ is the coefficient in \eqref{eq:sum-a} standing by $(H)$. %We also put
%\begin{equation}
%\label{eq:Phi(a)}\Phi(a):=\${(H)\in\Phi(G): \text{coeff}^{H}(a)\not =0\}.
%\end{equation}
%Clearly, the set $\Phi(a)$ is finite and we denote by $\text{max}(a)$ the set
%of all maximal (with respect to the order \eqref{eq:order}) elements in
%$\Phi(a)$, i.e.
%\begin{equation}
%\label{eq:max-in-a}\max(a):=\{(H)\in\Phi(a): \;(H) \text{ is maximal in }%
%\Phi(a)\}.
%\end{equation}
%We will sometimes write ${\max}^G(a)$ in order to indicate that the order relation is taken from $\Phi(G)$.
\vskip.3cm

\begin{definition}
\textrm{\label{def:EulerRing} (cf. \cite{tD}) Define the
\textit{multiplication} on $U(G)$ as follows: for generators $(H)$,
$(K)\in\Phi(G)$ put:
\begin{equation}
(H)\ast(K)=\sum_{(L)\in\Phi(G)}n_{L}(L),\label{eq:Euler-mult}%
\end{equation}
where
\begin{equation}
n_{L}:=\chi_{c}((G/H\times G/K)_{L}/N(L)),\label{eq:Euler-coeff}%
\end{equation}
Note,  $(G/H\times G/K)_{L}$ denotes the set of elements in $G/H\times G/K$ that are
fixed exactly by $L$. $\chi_c(\cdot)$ denotes the Euler Characteristic (For its precise definition, see Section 3 of \cite{DK} and \cite{Spa}).} Moreover, the multiplication is extended linearly to the entire Euler ring 
$U(G)$. Then the free ${\mathbb{Z}}$-module $U(G)$ associated with
multiplication (\ref{eq:Euler-mult}) is called the \textit{Euler ring} of $G$.
\end{definition}

\vskip.3cm It is easy to notice that $(G)$ is the unit element in $U(G)$, i.e.
$(G)*a=a$ for all $a\in U(G)$. \vskip.3cm

\begin{lemma}
\label{lem:inv-1} Assume that $a\in U(G)$ is an invertible element and
$(H)\in\Phi(G)$. Then
\[
\text{\text{\rm coeff}}^{H}((H)*a)\not =0.
\]

\end{lemma}

\begin{proof}
Suppose that
\[
a=\sum_{(L)\in\Phi(a)} n_{L} \, (L).
\]
Then
\[
(H)*a=\sum_{(K)\in\Phi((H)*a)} m_{K}\, (K), \text{ and formula
(\ref{eq:Euler-mult}) implies that} \;\; (H)\ge(K).
\]
Assume for contradiction that $(H)>(K)$ for all $(K)\in\Phi((H)*a)$. Then, by
exactly the same argument we have
\[
(H)*a*a^{-1}=\sum_{(L)\in\Phi((H)*a*a^{-1})} n_{L} \, (L), \quad\text{ where }
(H)>(L),
\]
which is a contradiction with the fact that
\[
(H)*a*a^{-1}=(H)*(G)=(H).
\]
\end{proof}
 %Lemma \ref{lem:inv-1} implies the following: \vskip.3cm

%\begin{corollary}
%\label{cor:inv} Let $a$, $b\in U(G)$ be such that $a$ is an invertible element
%in $U(G)$ and $(H)\in\text{\textrm{max}}(b)$. Then $\text{\textrm{coeff}}%
%^{H}(a*b)\not =0$.
%\end{corollary}

\vskip.3cm Take $\Phi_{0}(G)=\{(H)\in\Phi(G):$ \textrm{dim\thinspace
}$W(H)=0\}$ and denote by $A(G)={\mathbb{Z}}[\Phi_{0}(G)]$ a free
${\mathbb{Z}}$-module with basis $\Phi_{0}(G).$ Define multiplication on
$A(G)$ by restricting multiplication from $U(G)$ to $A(G)$, i.e. for $(H)$,
$(K)\in\Phi_{0}(G)$ let
\begin{equation}
(H)\cdot(K)=\sum_{(L)}n_{L}(L),\qquad(H),(K),(L)\in\Phi_{0}(G),\text{ where}
\label{eq:multBurnside}%
\end{equation}
\begin{equation}
n_{L}=\chi((G/H\times G/K)_{L}/N(L))=|(G/H\times G/K)_{L}/N(L)|
\label{eq:CoeffBurnside}%
\end{equation}
($\chi$ here denotes the usual Euler characteristic). Then $A(G)$ with
multiplication \eqref{eq:multBurnside} becomes a ring which is called the
\textit{Burnside ring} of $G$. As it can be shown, the coefficients
\eqref{eq:CoeffBurnside} can be found using the following
recursive formula:
\begin{equation}
n_{L}=\frac{n(L,K)|W(K)|n(L,H)|W(H)|-\sum_{(\widetilde{L})>(L)}n(L,\widetilde
{L})n_{\widetilde{L}}|W(\widetilde{L})|}{|W(L)|}, \label{eq:rec-coef}%
\end{equation}
where $(H),$ $(K),$ $(L)$ and $(\widetilde{L})$ are taken from $\Phi_{0}
(G),$ and \begin{align*}
&
N_{G}(L,H)=\Big\{
g\in G:gLg^{-1} \subset H\Big\} , \\
&N_{G}(L,H)/H=\Big\{
Hg: g\in N_{G}(L,H)\Big\}\\
& n(L,H)=\Big|\frac{N(L,H)}{N(H)}\Big|
\end{align*}
\vskip.3cm Observe that although $A(G)$ is clearly a ${\mathbb{Z}}$-submodule
of $U(G)$, in general, it is \textbf{not} a subring of $U(G)$.

\vskip.3cm Define $\pi_{0}:U(G)\rightarrow A(G)$ as follows: for $(H)\in
\Phi(G)$ let
\begin{equation}
\pi_{0}((H))=
\begin{cases}
(H) & \text{ if }\;(H)\in\Phi_{0}(G),\\
0 & \text{ otherwise.}%
\end{cases}
\label{eq:pi_0-homomorphism}%
\end{equation}
The map $\pi_{0}$ defined by (\ref{eq:pi_0-homomorphism}) is a ring
homomorphism {(cf. \cite{BKR})}, i.e.
\[
\pi_{0}((H)\ast(K))=\pi_{0}((H))\cdot\pi_{0}((K)),\qquad(H),(K)\in\Phi(G).
\]
The following well-known result (cf. \cite{tD}, Proposition 1.14, page 231)
shows a difference between the generators $(H)$ of $U(G)$ and $A(G)$.

\begin{proposition}
\label{pro:nilp-elements} Let $(H)\in\Phi_{n}(G)$.

\begin{itemize}
\item[(i)] If $n>0$, then $(H)^{k}=0$ in $U(G)$ for some $k\in\mathbb{N}$,
i.e. $(H)$ is a nilpotent element in $U(G)$;

\item[(ii)] If $n=0$, then $(H)^{k}\not =0$ for all $k\in{\mathbb{N}}$.
\end{itemize}
\end{proposition}
\begin{comment}
\begin{proof}
Let $n\geq 1$ be given. Consider 
\[
X=(G/H)^n=G/H\times G/H\times\cdots \times G/H,
\]
and notice that $\Phi(X;G)$ is finite, and 
\[
(H)^n=\sum_{(K)}\alpha_K(K),\;\;\; (K)\in \Phi(X;G)
\]
choose $(L)\in \Phi(X;G)$ such that $(L)$ is maximal among all $(K)\in \Phi(X;G)$ such that $\alpha_K\neq 0.$ Then
\begin{align*}
(H)^{k+1}&=(H)^k
\cdot (H)=\sum_{K}(K)(H)\\
&=\sum_{(K')}\beta_{K'}(K')
\end{align*}
and $\beta_L$ can occur only in the product $\alpha_L(L)\cdot (H),$ which implies
\begin{align*}
\beta_L&=\chi_c\Big((G/H\times G/L)_{(L)}/G\Big)=\chi_c\Big((G/H\times G/L)_{L}/N(L)\Big)\\
&=\chi\Big((G/H\times G/L)^{L}/N(L)\Big)=\chi\Big((G/H)^L\times W(L))/W(L)\Big)\\
&=\chi\Big((G/H)^L\Big)
\end{align*}
Since $W(H)$ acts freely on $(G/H)^L=\frac{N(L,H)}{H}$ and $\dim W(H)=\dim W(H)/H\ge 0,$ then $S^1\subset W(H)$. Since $S^1 $ also acts freely on $(G/H)^L,$ we have
\[
\chi \Big((G/H)^L\Big)=0
\]
\end{proof}
\end{comment}
\begin{corollary}
    If $\alpha=n_1(L_1)+n_2(L_2)+\cdots +n_k(L_k),$ where $\dim W(L_j)\ge 0,$ then there exists $n\in \mathbb N$ s.t. $\alpha^n=0.$
\end{corollary}
\begin{proof}
    By induction w.r.t. $k\in \mathbb N$, clearly for $k=1,$ it is exactly the statement of proposition\eqref{pro:nilp-elements}. Suppose that the statement is true for $k\geq 1,$ and will show that it is also true for $k+1.$ Indeed, we have
    \[
\alpha^1=n_1(L_1)+n_2(L_2)+\cdots n_k(L_k)+n_{k+1}(L_{k+1})
    =\alpha+n_{k+1}(L_{k+1}),
    \]
so \[
\alpha^m=\sum_{l=0}^{m}C_m^l \alpha^l n_k^{m-l}(L_{k+1})^{m-l}.
\]
Let $k$ be given by Proposition \eqref{pro:nilp-elements}, for $L:=l_{k+1},$ then for $m\ge n+k,$ one has 
\[
\alpha^l (L)^{m-l}=0
\]
    \end{proof}
\vskip.3cm 
\black
\section{Properties of $G$-Equivariant Gradient Degree}

\label{app:D}

In what follows, we assume that $G$ is a compact Lie group.

\subsection{Brouwer $G$-Equivariant Degree}

Assume that $V$ is an orthogonal $G$-representation and $\Omega\subset V$ an
open bounded $G$-invariant set. A $G$-equivariant (continuous) map $f:V\to V$
is called \textit{$\Omega$-admissible} if $f(x)\neq0$ for any $x\in
\partial\Omega$; in such a case, the pair $(f,\Omega)$ is called
\textit{$G$-admissible}. Denote by $\mathcal{M}^{G}(V,V)$ the set of all such
admissible $G$-pairs, and put $\mathcal{M}^{G}:=\bigcup_{V}\mathcal{M}
^{G}(V,V)$, where the union is taken for all orthogonal $G$-representations
$V$. We have the following result:

\begin{definition}
\label{thm:GpropDeg} There exists a unique map $G$-$\deg:\mathcal{M}^{G}\to
A(G)$, which assigns to every admissible $G$-pair $(f,\Omega)$ an element
$G\mbox{\rm -}\deg(f,\Omega)\in A(G)$, called the \textit{$G$-equivariant
degree} (or simply \textit{$G$-degree}) of $f$ on $\Omega$:
\begin{equation}
\label{eq:G-deg0}G\mbox{\rm -}\deg(f,\Omega)=\sum_{(H_{i})\in\Phi_{0}
(G)}n_{H_{i}}(H_{i})= n_{H_{1}}(H_{1})+\dots+n_{H_{m}}(H_{m}).
\end{equation}
It satisfies the properties of additivity, homotopy, normalization, as well as existence, product, suspension, recurrence formula, etc. (see \cite{balanov2025degree} for details).
We call $G\mbox{\rm -}\deg(f,\Omega)$ the \textit{$G$-equivariant
degree} (or simply \textit{$G$-degree}) of $f$ on $\Omega$.
\end{definition}\vs 
\begin{definition} The Brouwer $G$-equivariant degree
\begin{equation} \label{eq:basicdeg}
\deg_{\cV_i} := G\text{-}\deg(-\id, B(\cV_i)),
\end{equation}
is called the {\it $\cV_i$-basic degree} (or simply {\it basic degree}), and it can be computed by: $\deg_{\mathcal{V} _{i}}=\sum_{(K)}n_{K}(K),$
where
\begin{align}\label{eq:formula}
n_{K}=\frac{(-1)^{\dim\mathcal{V} _{i}^{K}}- \sum_{K<L}{n_{L}\, n(K,L)\, \left|  W(L)\right|  }}{\left|  W(K)\right|  }.
\end{align}
\end{definition}

\begin{lemma}\label{le:basic_coefficient}
If for $ (K_o) \in \Phi_0(G)$, one has $\text{coeff}^{L_o} (\deg_{\cV_i})$ is a leading coefficient of $\deg_{\cV_i}$, then $\dim (\cV_i ^{K_o})$ is odd and
\[
\text{coeff}^{K_o}(\deg_{\cV_i})=
\begin{cases}\label{eq:a0}
-1 & \text{if} \; |W(K_o)|=2,\\
-2 & \text{if}\; |W(K_o)|=1;\\
\end{cases}
\]
\end{lemma}
\begin{lemma}\label{le:involutive}
    For each  $\mathcal V_i$,  the corresponding basic degree $\deg_{\mathcal V_i} \in A(G)$ is an involution in the Burnside ring. It satisfies
    \[
    (\deg_{\cV_i})^2=\deg_{\cV_i} \cdot \deg_{\cV_i}=(G).
    \]\vs
\end{lemma}
\vskip.3cm

%---

\subsection{$G$-Equivariant Gradient Degree}

%---
Let $V$ be an orthogonal $G$-representation. Denote by $C^{2}_{G}
(V,\mathbb{R})$ the space of $G$-invariant real $C^{2}$-functions on $V$. Let
$\varphi\in C^{2}_{G}(V,\mathbb{R})$ and $\Omega\subset V$ be an open bounded
invariant set such that $\nabla\varphi(x)\not =0$ for $x\in\partial\Omega$. In
such a case, the pair $(\nabla\varphi,\Omega)$ is called \textit{$G$-gradient
$\Omega$-admissible}. Denote by $\mathcal{M}^{G}_{\nabla}(V,V)$ the set of all
$G $-gradient $\Omega$-admissible pairs in $\mathcal{M}^{G}(V,V)$ and put
$\mathcal{M}^{G}_{\nabla}:=\bigcup_{V}\mathcal{M}^{G}_{\nabla}(V,V)$. 

\begin{theorem}
\label{thm:Ggrad-properties} There exists a unique map
$\nabla_{G}\mbox{\rm -}\deg:\mathcal{M}_{\nabla}^{G}\to U(G)$, which assigns
to every $(\nabla\varphi,\Omega)\in\mathcal{M}^{G}_{\nabla}$ an element
$\nabla_{G}\mbox{\rm -}\deg(\nabla\varphi,\Omega)\in U(G)$, called the
\textit{$G$-gradient degree} of $\nabla\varphi$ on $\Omega$,
\begin{equation}
\label{eq:grad-deg}\nabla_{G}\mbox{\rm -}\deg(\nabla\varphi,\Omega)=
\sum_{(H_{i})\in\Phi(\Gamma)}n_{H_{i}}(H_{i})= n_{H_{1}}(H_{1})+\dots
+n_{H_{m}}(H_{m}),
\end{equation}
satisfying the following properties:

\begin{enumerate}
%\item \renewcommand\labelenumi{\rm\bf($\nabla$\arabic{enumi})}

\item \textrm{\textbf{(Existence)}} If $\nabla_{G}\mbox{\rm -}\deg
(\nabla\varphi,\Omega)\not =0$, i.e., \eqref{eq:grad-deg} contains a non-zero
coefficient $n_{H_{i}}$, then $\exists_{x\in\Omega}$ such that $\nabla
\varphi(x)=0$ and $(G_{x})\geq(H_{i})$.

\item \textrm{\textbf{(Additivity)}} Let $\Omega_{1}$ and $\Omega_{2}$ be two
disjoint open $G$-invariant subsets of $\Omega$ such that $(\nabla
\varphi)^{-1}(0)\cap\Omega\subset\Omega_{1}\cup\Omega_{2}$. Then,
\[
\nabla_{G}\mbox{\rm -}\deg(\nabla\varphi,\Omega)= \nabla_{G}\mbox{\rm -}\deg
(\nabla\varphi,\Omega_{1})+ \nabla_{G}\mbox{\rm -}\deg(\nabla\varphi
,\Omega_{2}).
\]

\item \textrm{\textbf{(Homotopy)}} If $\nabla_{v}\Psi:[0,1]\times V\to V$ is a
$G$-gradient $\Omega$-admissible homotopy, then
\[
\nabla_{G}\mbox{\rm -}\deg(\nabla_{v}\Psi(t,\cdot),\Omega)=\mathrm{constant}.
\]

\item \textrm{\textbf{(Normalization)}} Let $\varphi\in C^{2}_{G}
(V,\mathbb{R})$ be a special $\Omega$-Morse function such that $(\nabla
\varphi)^{-1}(0)\cap\Omega=G(v_{0})$ and $G_{v_{0}}=H$. Then,
\[
\nabla_{G}\mbox{\rm -}\deg(\nabla\varphi,\Omega)= (-1)^{{m}^{-}(\nabla
^{2}\varphi(v_{0}))}\cdot(H),
\]
where ``$m^{-}(\cdot)$'' stands for the total dimension of eigenspaces for
negative eigenvalues of a (symmetric) matrix.

\item \textrm{\textbf{(Product)}} For all $(\nabla\varphi_{1},\Omega
_{1}),(\nabla\varphi_{2},\Omega_{2})\in\mathcal{M}^{G}_{\nabla}$,
\[
\nabla_{G}\mbox{\rm -}\deg(\nabla\varphi_{1}\times\nabla\varphi_{2},
\Omega_{1}\times\Omega_{2})= \nabla_{G}\mbox{\rm -}\deg(\nabla\varphi
_{1},\Omega_{1})\ast\nabla_{G}\mbox{\rm -}\deg(\nabla\varphi_{2},\Omega_{2}),
\]
where the multiplication `$\ast$' is taken in the Euler ring $U(G)$.

%\item \textrm{\textbf{(Suspension)}} If $W$ is an orthogonal $G$
%-representation and $\mathcal{B}$ an open bounded invariant neighborhood of
%$0\in W$, then
%\[
%\nabla_{G}\mbox{\rm -}\deg(\nabla\varphi\times\text{\text{\textrm{Id}}}_{W}
%,\Omega\times\mathcal{B})= \nabla_{G}\mbox{\rm -}\deg(\nabla\varphi,\Omega).
%\]

%\item \textrm{\textbf{(Hopf Property)}} Assume $B(V)$ is the unit ball of an
%orthogonal $\Gamma$-representation $V$ and for $(\nabla\varphi_{1}
%,B(V)),(\nabla\varphi_{2},B(V))\in\mathcal{M}^{G}_{\nabla}$, one has
%\[
%\nabla_{G}\mbox{\rm -}\deg(\nabla\varphi_{1},B(V))= \nabla_{G}\mbox{\rm -}\deg
%(\nabla\varphi_{2},B(V)).
%\]
%Then, $\nabla\varphi_{1}$ and $\nabla\varphi_{2}$ are $G$-gradient
%$B(V)$-admissible homotopic.

%\item \textrm{\textbf{(Functoriality Property)}} Let $V$ be an orthogonal
%$G$-representation, $f:V\rightarrow V$ a $G$-gradient $\Omega$-admissible map,
%and $\psi:G_{0}\hookrightarrow G$ an embedding of Lie groups. Then, $\psi$
%induces a $G_{0}$-action on $V$ such that $f$ is an $\Omega$-admissible
%$G_{0}$-gradient map, and the following equality holds
%\begin{equation}
%\Psi\lbrack\nabla_{G}\mbox{\rm -}\deg(f,\Omega)]=\nabla_{G_{0}}
%\mbox{\rm -}\deg(f,\Omega), \label{eq:funct-G}%
%\end{equation}
%where $\Psi:U(G)\rightarrow U(G_{0})$ is the homomorphism of Euler rings
%induced by $\psi$.

\item \textrm{\textbf{(Reduction Property)}} Let $V$ be an orthogonal
$G$-representation, $f:V\to V$ a $G$-gradient $\Omega$-admissible map, then
\begin{equation}
\label{eq:red-G}\pi_{0}\left[  \nabla_{G}\mbox{\rm -}\deg(f,\Omega)\right]
=G\mbox{\rm -}\deg(f,\Omega).
\end{equation}
where the ring homomorphism $\pi_{0}:U(G)\to A(G)$ is given by \eqref{eq:pi_0-homomorphism}.
\end{enumerate}
\end{theorem}
For other properties such as Functoriality, Hopf Property, Suspension, etc., one is referred to Section 6 of \cite{DK}.

\vs 

\subsection{Computations of the Gradient $G$-Equivariant Degree}

%-
Similarly to the case of the Brouwer degree, the gradient equivariant degree
can be computed using standard linearization techniques. Therefore, it is
important to establish computational formulae for linear gradient operators.
\vskip.3cm Let $V$ be an orthogonal (finite-dimensional) $G$-representation
and suppose that $A:V\to V$ is a $G$-equivariant symmetric isomorphism of $V$,
i.e., $A:=\nabla\varphi$, where $\varphi(x)=\frac12 Ax\bullet x$. Consider the
$G$-isotypical decomposition of $V$
\[
V=\bigoplus_{i}V_{i},\quad V_{i}\;\;\mbox{modeled on}\;\mathcal{V}_{i}.
\]
We assume here that $\{\mathcal{V}_{i}\}_{i}$ is the complete list of
irreducible $G$-representations. \vskip.3cm Let $\sigma(A)$ denote the
spectrum of $A$ and $\sigma_{-}(A):=\{\lambda\in\sigma(A): \lambda<0\}$, and
let $E_{\mu}(A)$ stands for the eigenspace of $A$ corresponding to $\mu
\in\sigma(A)$. Put $\Omega:=\{x\in V: \|x\|<1\}$. Then, $A$ is $\Omega
$-admissibly homotopic (in the class of gradient maps) to a linear operator
$A_{o}:V\to V$ such that
\[
A_{o}(v):=
\begin{cases}
-v, & \mbox{if}\;v\in E_{\mu}(A),\;\mu\in\sigma_{-}(A),\\
v, & \mbox{if}\;v\in E_{\mu}(A),\;\mu\in\sigma(A)\setminus\sigma_{-}(A).
\end{cases}
\]
In other words, $A_{o}|_{E_{\mu}(A)}=-\text{\textrm{Id}}$ for $\mu\in
\sigma_{-}(A)$ and $A_{o}|_{E_{\mu}(A)}=\text{\textrm{Id}}$ for $\mu\in
\sigma(A)\setminus\sigma_{-}(A)$. Suppose that $\mu\in\sigma_{-}(A)$. Then,
denote by $m_{i}(\mu)$ the integer
\[
m_{i}(\mu):=\dim(E_{\mu}(A)\cap V_{i})/\dim\mathcal{V}_{i},
\]
which is called the \textit{$\mathcal{V}_{i}$-multiplicity} of $\mu$. Since
$\nabla_{G}\mbox{\rm -}\deg(\text{\textrm{Id}},\mathcal{V}_{i})=(G)$ is the
unit element in $U(G)$, we immediately obtain, by product property ($\nabla
$5), the following formula
\begin{equation}
\label{eq:grad-lin}\nabla_{G}\mbox{\rm -}\deg(A,\Omega)= \prod_{\mu\in
\sigma_{-}(A)}\prod_{i} \left[  \nabla_{G}\mbox{\rm -}\deg(-\text{\textrm{Id}
},B(\mathcal{V}_{i}))\right]  ^{m_{i}(\mu)},
\end{equation}
where $B(W)$ is the unit ball in $W$. \vskip.3cm

\begin{definition} \label{de:basicequi}
\textrm{\ Assume that $\mathcal{V}_{i}$ is an irreducible $G$-representation.
Then, the $G$-equivariant gradient degree:
\[
\nabla_{G}\text{-}\deg_{\mathcal{V}_{i}}:=\nabla_{G}\mbox{\rm -}\deg(-\text{\text{\rm Id}}
,B(\mathcal{V}_{i}))\in U(G)
\]
is called the \textit{gradient $G$-equivariant basic degree} for
$\mathcal{V}_{i}$. }
\end{definition}

\vskip.3cm

\begin{proposition}
\label{prop:invertible} The gradient $G$- equivariant basic degrees $\nabla_{G}\text{-}\deg_{\mathcal{V}_{i}}$ are invertible elements in $U(G)$.
\end{proposition}

\begin{proof}
Let $a:=\pi_{0}(\nabla_G\text{\textrm{-deg\,}}_{\cV_{i}})$, then $a^{2}=(G)$ in $A(G)$ (see Lemma \eqref{le:involutive}),
which implies that $(\nabla_{G}\text{\textrm{-deg\,}}_{\mathcal{V}_{i}})^{2}=(G)-\alpha$,
where for every $(H)\in\Phi_{0}(G)$ one has coeff$^{H}(\alpha)=0$. It is
sufficient to show that $(G)-\alpha$ is invertible in $U(G)$. Since (by
Proposition \ref{pro:nilp-elements}) for sufficiently large $n\in{\mathbb{N}}%
$, $\alpha^{n}=0$, one has
\[
\big((G)-\alpha)\sum_{n=0}^{\infty}\alpha^{n}=\sum_{n=0}^{\infty}\alpha
^{n}-\sum_{n=1}^{\infty}\alpha^{n}=(G),
\]
where $\alpha^{0}=(G)$
\end{proof}

\vs
\vskip.3cm
\black
\vskip.3cm \noindent\textbf{Degree on the Slice:} Let $\mathscr H$ be a
Hilbert $G$-representation. Suppose that the orbit $G(u_{o})$ of $u_{o}
\in\mathscr H$ is contained in a finite-dimensional $G$-invariant subspace, so
the $G$-action on that subspace is smooth and $G(u_{o})$ is a smooth
submanifold of $\mathscr H$. In such a case we call the orbit $G(u_{o})$
\textit{finite-dimensional}. Denote by $S_{o}\subset\mathscr H$ the slice to
the orbit $G(u_{o})$ at $u_{o}$. Denote by $V_{o}:=\tau_{u_{o}}G(u_{o})$ the
tangent space to $G(u_{o})$ at $u_{o}$. Then $S_{o}=V_{o}^{\perp}$ and $S_{o}
$ is a smooth Hilbert $G_{u_{o}}$-representation.

Then we have (cf. \cite{BeKr}):

\begin{theorem}
\textrm{(Slice Principle)} \label{thm:SCP} Let $\mathscr{H}$ be a Hilbert
$G$-representation, $\Omega$ an open $G$-invariant subset in $\mathscr H$, and
$\varphi:\Omega\rightarrow\mathbb{R}$ a continuously differentiable
$G$-invariant functional such that $\nabla\varphi$ is a completely continuous
field. Suppose that $u_{o}\in\Omega$ and $G(u_{o})$ is an finite-dimensional
isolated critical orbit of $\varphi$ with $S_{o}$ being the slice to the orbit
$G(u_{o})$ at $u_{o}$, and $\mathcal{U}$ an isolated tubular neighborhood of
$G(u_{o})$. Put $\varphi_{o}:S_{o}\rightarrow\mathbb{R}$ by $\varphi
_{o}(v):=\varphi(u_{o}+v)$, $v\in S_{o}$. Then
\begin{equation}
\nabla_{G}\text{\text{\rm-deg\,}}(\nabla\varphi,\mathcal{U})=\Theta
(\nabla_{G_{u_{o}}}\text{\text{\rm-deg\,}}(\nabla\varphi_{o},\mathcal{U}\cap
S_{o})), \label{eq:SDP}%
\end{equation}
where $\Theta:U(G_{u_{o}})\rightarrow U(G)$ is homomorphism defined on
generators $\Theta(H)=(H)$, $(H)\in\Phi(G_{u_{o}})$.
\end{theorem}
\section{Block matrix and Hessian $\nabla^2 U(v^o)$} \label{app:pre}
\subsection{Block matrix $P_{jk}$.} From equation \eqref{eqn: Pjk}, one has

\[
\begin{aligned}
&P_{11} = P_{33} =2a(\mathfrak m_{12}+\mathfrak m_{14}+\mathfrak m_{15}+\mathfrak m_{16})+2b\mathfrak m_{13}+2c\mathfrak m_{10}-d\id=
\begin{bmatrix}
    8a+8b+2c-d & 0 & 0 \\ 
    0 & 4a-d & 0 \\ 
    0 & 0 & 4a-d
\end{bmatrix}, \\
& P_{22} = P_{44} =2a(\mathfrak m_{12}+\mathfrak m_{23}+\mathfrak m_{25}+\mathfrak m_{26})+2b\mathfrak m_{24}+2c\mathfrak m_{20}-d\id=
\begin{bmatrix}
    4a-d & 0 & 0 \\ 
    0 & 8a+8b+2c-d & 0 \\ 
    0 & 0 & 4a-d
\end{bmatrix}, \\ 
& P_{55} = P_{66} =2a(\mathfrak m_{15}+\mathfrak m_{25}+\mathfrak m_{35}+\mathfrak m_{45})+2b\mathfrak m_{56}+2c\mathfrak m_{50}-d\id=
\begin{bmatrix}
    4a-d & 0 & 0 \\ 
    0 & 4a-d & 0 \\ 
    0 & 0 & 8a+8b+2c-d
\end{bmatrix}, \\
& P_{12}=P_{34}=-2a\mathfrak m_{12}-e\id=
\begin{bmatrix}
    -2a-e&2a&0\\
    2a&-2a-e&0\\
    0&0&-e
\end{bmatrix},  P_{14}=P_{23}=-2a\mathfrak m_{14}-e\id=
\begin{bmatrix}
    -2a-e&-2a&0\\
    -2a&-2a-e&0\\
    0&0&-e
\end{bmatrix}, \\ 
&P_{15}=P_{36}=-2a\mathfrak m_{15}-e\id=
\begin{bmatrix}
    -2a-e&0&2a\\
    0&-e&0\\
    2a&0&-2a-e
\end{bmatrix}, P_{16}=P_{35}=-2a\mathfrak m_{16}-e\id=
\begin{bmatrix}
    -2a-e&0&-2a\\
    0&-e&0\\
    -2a&0&-2a-e
\end{bmatrix},\\
\end{aligned}
\]
\[
\begin{aligned}
& P_{25}=P_{46}=-2a\mathfrak m_{25}-e\id=
\begin{bmatrix}
-e&0&0\\
    0&-2a-e&2a\\
    0&2a&-2a-e\\
\end{bmatrix},P_{26}=P_{45}=-2a\mathfrak m_{26}-e\id=
\begin{bmatrix}
-e&0&0\\
    0&-2a-e&-2a\\
    0&-2a&-2a-e\\
\end{bmatrix},\\
& P_{13}=-2b\mathfrak m_{13}-d\id=
\begin{bmatrix}
    -8b-d&0&0\\
    0&-d&0\\
    0&0&-d
\end{bmatrix}, 
P_{24}=-2b\mathfrak m_{24}-d\id=
\begin{bmatrix}
    -d&0&0\\
    0&-8b-d&0\\
    0&0&-d
\end{bmatrix}, \\
\end{aligned}
\]
\[
\begin{aligned}
& P_{56}=-2b\mathfrak m_{56}-d\id=
\begin{bmatrix}
    -d&0&0\\
    0&-d&0\\
    0&0&-8b-d
\end{bmatrix},\\ 
\end{aligned}
\]\vs
\begin{comment}
\subsection{Nondimensionalization of parameters $\sigma_1, \sigma_2,\sigma_3.$} From \cite{}, one has
\begin{align*}
& \epsilon_F/_{k_B}=26.68[k]\implies \epsilon_F=26.68\times 1.380649\times 10^{-23}=3.6836\times 10^{-22}J,\\
&\sigma_F=2.963 \overset{\circ}{A}=2.963\times 10^{-10}m,\\
&q_F=-0.175\times 1.602176634\times 10^{-19}=-2.8038\times 10^{-20}C,
\end{align*}
The parameters $\sigma_1,\sigma_2,\sigma_3$ in equation \eqref{eq:U1} has the following form: $
\sigma_1=4\epsilon_F\sigma^{12}_F, \sigma_2=4\epsilon_F\sigma_F^{6},\sigma_3=\frac{q_F^2}{4\pi\epsilon_0},$
and we pick the characteristic scale for nondimensionalization: $L_0=\sigma_F$ and  $E_0=\frac{q^2_F}{4\pi\epsilon_0\sigma_F}.$
Therefore,
\begin{align*}
& \sigma_1:=\frac{\sigma_1}{E_0L^{12}_0}=\frac{4\epsilon_F \sigma^{12}_F}{（q^2_F/r\pi\epsilon_0\sigma_F）\cdot \sigma^{12}_F}=\frac{16\pi\epsilon_0\epsilon_F\sigma_F}{q^2_F}=0.0617906\\
& \sigma_2:=\frac{\sigma_1}{E_0L^{6}_0}=\frac{4\epsilon_F \sigma^{6}_F}{（q^2_F/r\pi\epsilon_0\sigma_F）\cdot \sigma^{6}_F}=\frac{16\pi\epsilon_0\epsilon_F\sigma_F}{q^2_F}=0.0617906\\
&\sigma_3:=\frac{\sigma_3}{E_0L_0}=\frac{q^2_F /4\pi\epsilon_0}{（q^2_F/4\pi\epsilon_0\sigma_F）\cdot \sigma_F}=1\\
\end{align*}
\end{comment}
%\begin{comment}
\subsection{Eigenspace corresponding to Hessian $\nabla^2 U(v^o)|_{v^o=1.41278p^o}$.}
By direct computation, one has the following detailed eigensystem of $\nabla^2 U(v^o)_{v^o=1.41278p^o}$ corresponding to Table \eqref{tbl:eigen}:

\begin{table}[htbp]
\centering
\rotatebox{-90}{
\adjustbox{ width=1.5\textwidth}{
\begin{tabular}{|l|c|l|}
\hline
Eigenvalue $\alpha^2_j$ at $v^o$ & Irreducible rep $\mathcal W_j$& Eigenvectors \\ \hline
 $\alpha^2_0=2(8a+8b+c)=0.7867$ & $\mathcal W_0$& \makecell[l]{$v_0^1=0.408(1, 0, 0,| 0, 1, 
0,| -1, 0, 0,|
0, -1, 0,| 0, 
0, 1,|0, 0, -1)^T$}\\ \hline
$\alpha^2_4=2(2a+8b+c)=0.5123$&$\mathcal W_4$ & \makecell[l]{$v_4^1=(0.2887, 0, 0,| 0, 0.2887,
0,| -0.2887, 0, 0,| 
0, -0.2887, 0,| 0, 
0, -0.577,| 0, 0, 0.577)^T$\\
$v_4^2=(0.5, 0, 0,| 0, -0.5,
0,| -0.5, 0, 0,| 0,
0.5, 0,| 0, 0, 0,|
0, 0, 0)^T$} \\ \hline
%$\alpha^2_6=0$&$\mathcal W_6$& 3 \\ \hline
$\alpha^2_7=6a+c-2d-e-\rho=0.2532$&$\mathcal W_7$& \makecell[l]{$v_7^1=(0.0182, -0.0376, -0.485,| -0.0919, 0.0074, -0.485,|
0.0182, -0.0376, -0.485,| -0.0919, 0.0074, -0.485,| 
-0.0919, -0.0376, 0.0959,| -0.0919, -0.0376, 0.0959)^T$\\
$v_7^2=(0.006791, -0.492, 0.0446,| -0.0343, 0.0973, 0.0446,|
0.006791, -0.492, 0.0446,| -0.0343, 0.0973, 0.0446,|
-0.0343, -0.492, -0.0088,| -0.0343, -0.492, -0.0088)^T$\\
$v_7^3=(0.096, 0.0419, 0.0888,| -0.485, -0.0083, 0.0888,|
0.096, 0.0419, 0.0888,| -0.485, -0.0083, 0.0888,|
-0.485, 0.0419, -0.0176,| -0.485, 0.0419, -0.0176)^T$} \\ \hline
$\alpha^2_{7^*}= 6a+c-2d-e+\rho=0.5882$&$\mathcal W_7^*$ &\makecell[l]{$v_{7^*}^1=(0.00739, 0.00042, 0.0693,| 0.000731, 0.00427,
0.0693,| 0.00739, 0.00042, 0.0693,| 0.000731, 
0.00427, 0.0693,| 0.000731, 0.00042, 0.7002,| 
0.000731, 0.00042, 0.7002)^T$\\
$v_{7^*}^2=(-0.599, -0.0358, 0.00084,| -0.0593, -0.362, 
0.00084,| -0.599, -0.0358, 0.00084,| -0.0593, 
-0.362, 0.00084,| -0.0593, -0.0358, 0.0085,|
-0.0593, -0.0358, 0.00854)^T$\\
$v_{7^*}^3=(0.362, -0.0593, -0.000016,| 0.0358, -0.599, 
-0.000016,| 0.362, -0.0593, -0.000016,| 0.0358, 
-0.599, -0.000016,| 0.0358, -0.0593, -0.000167,| 
0.0358, -0.0593, -0.000167)$} \\ \hline
$\alpha^2_8= 8a=0.1829$ & $\mathcal W_8$ &\makecell[l]{
$v_8^1=(0, 0.4068, 0.166,| 0.4068, 0, 
-0.239,| 0, -0.4068, -0.166,| -0.4068, 
0, 0.239,| 0.166, -0.239, 0,| 
-0.166, 0.239,0)^T$\\
$v_8^2=(0, 0.291, -0.222,| 0.291, 0, 
0.341,| 0, -0.291, 0.222,| -0.291, 
0, -0.341,| -0.222, 0.341, 0,| 
0.222, -0.341, 0)^T$\\
$v_8^3=(0, -0.00699, 0.416,| -0.00699, 0, 0.2772,| 0, 0.00699, -0.416,| 0.00699, 
0, -0.2772,| 0.416, 0.2772, 0,| 
-0.416, -0.2772, 0)^T$
} \\ \hline
$\alpha^2_9= 2(2a-d+e)=0.01173$ &$\mathcal W_9$& \makecell[l]{
$v_9^1=(0, -0.1726, 0.266,| -0.3865, 0, 
-0.266,| 0, -0.1726, 0.266,| -0.3865, 
0, -0.266,| 0.3865, 0.1726, 0,| 
0.3865, 0.1726, 0)^T$\\
$v_9^2=(0, 0.06195, 0.4212,| 0.2622, 0, \
-0.4212,| 0, 0.06195, 0.4212,| 0.2622, \
0, -0.4212,| -0.2622, -0.06195, 0,| \
-0.2622, -0.06195, 0)^T$\\
$v_9^3=(0, -0.465, -0.0426,| 0.1784, 0, \
0.0426,| 0, -0.465, -0.0426,| 0.1784, \
0, 0.0426,| -0.1784, 0.465, 0,| 
-0.1784, 0.465, 0)^T$
} \\ \hline
\end{tabular}}
}
\caption{Eigensystem of Hessian $\nabla^2U(v^o)$}
\label{tbl:eigensystem}
\end{table}
\newpage

\end{document}